\documentclass[11pt]{article}
\usepackage{latexsym}
\usepackage{amsmath,amsthm,amsfonts,amssymb}
\usepackage[T1]{fontenc}
\usepackage[latin1]{inputenc}
\usepackage{aeguill} 
\usepackage{url} 

\usepackage[shortlabels]{enumitem}

\usepackage{graphicx}
\usepackage[a4paper,textwidth=15cm,textheight=23.1cm]{geometry}
\usepackage{xspace}
\usepackage{export}

\newtheorem{thm}{Theorem}[section]

\newtheorem*{thm*}{Theorem}
\newtheorem{dfn}[thm]{Definition} 
\newtheorem*{dfn*}{Definition}

\newtheorem{cor}[thm]{Corollary}
\newtheorem*{cor*}{Corollary}

\newtheorem{prop}[thm]{Proposition} 
\newtheorem*{prop*}{Proposition} 
\newtheorem*{properties*}{Properties} 
 
\newtheorem{lem}[thm]{Lemma} 
\newtheorem*{lem*}{Lemma}

\newtheorem*{claim*}{Claim} 
 
\newtheorem*{fact*}{Fact}

\newtheorem*{qst*}{Question}

\newtheorem*{pb*}{Problem}

\newtheorem{con}[thm]{Convention}

\theoremstyle{remark}
 
\newtheorem*{algo*}{Algorithm} 
\newtheorem*{rem*}{Remark}
\newtheorem{rem}[thm]{Remark}
\newtheorem*{example*}{Example}

\newcounter{numEnonceTmpInterne}\newenvironment{enonce*}[1]{\theoremstyle{plain}\stepcounter{numEnonceTmpInterne} 
\def\a{enoncetmp\alph{numEnonceTmpInterne}} 
\newtheorem*{\a}{#1}\begin{\a}}{\end{\a}}

\makeatletter
\edef\@tempa#1#2{\def#1{\mathaccent\string"\noexpand\accentclass@#2 }}
\@tempa\rond{017}
\makeatother

\newcommand{\es}{\emptyset}
\renewcommand{\phi}{\varphi} 
\newcommand{\m} {^{-1}}

\newcommand {\ra} {\rightarrow}

\newcommand {\xra} {\xrightarrow}

\renewcommand{\subsetneq}{\varsubsetneq}

\newcommand{\ie} {i.e.\ }

\newcommand {\cala} {{\mathcal {A}}}   
   
\newcommand {\calc} {{\mathcal {C}}}   
\newcommand {\cald} {{\mathcal {D}}}

\newcommand {\calg} {{\mathcal {G}}}   
\newcommand {\calh} {{\mathcal {H}}}

\newcommand {\calk} {{\mathcal {K}}}   
\newcommand {\call} {{\mathcal {L}}}   
\newcommand {\calm} {{\mathcal {M}}}

\newcommand {\calp} {{\mathcal {P}}}   
\newcommand {\calq} {{\mathcal {Q}}}   
\newcommand {\calr} {{\mathcal {R}}}   
   
\newcommand {\calt} {{\mathcal {T}}}   
\newcommand {\calu} {{\mathcal {U}}}

\newcommand {\calz} {{\mathcal {Z}}}

\newcommand {\bG} {{\mathbf {G}}}

\newcommand {\bQ} {{\mathbf {Q}}}

\newcommand {\bbZ} {{\mathbb {Z}}}

\newcommand{\grp}[1]{\langle #1 \rangle}

\newcommand{\Stab} {{\mathrm{Stab}}}
\newcommand{\Fix}{{\mathrm{Fix\,}}}

\newcommand{\Out} {{\mathrm{Out}}}

\newcommand{\Aut} {{\mathrm{Aut}}}

\newcommand{\Zmax}{\calz_{\mathrm{max}}}

\newcommand{\Inc}{\mathrm{Inc}}

\newcommand{\M}{{\mathrm{Mc}}}
\newcommand{\AM}{{\mathrm{Ac}}}

\newcommand {\V} {{\mathcal {V}}}

\newcommand{\RC} {RC}

\newcommand{\wh} {\widehat}
\newcommand{\hp} {\calh^{+ab}}
\newcommand{\cip} {\calc_i^{+ab}}

\usepackage{srcltx}
\usepackage[all]{xypic}

\newcommand{\Tcan}{T_{\mathrm{can}}}
\newcommand{\Gcan}{\Gamma_{\mathrm{can}}}
\newcommand{\Z}{{\mathbb {Z}}}

\usepackage{pdfsync}
\newcommand{\inc}{\subset}
\newcommand{\incs}{\subsetneq}
\newcommand{\incd}{\supsetneq}

\newcommand{\Rt}{$\R$-tree}

\newcommand {\R} {{\mathbb {R}}}
\newcommand {\Q} {{\mathbb {Q}}}

\newcommand{\Tw}{\calt}

\newcommand{\mk}[2][t]{{#2}^{\mathrm{(#1)}}}
 
\renewcommand {\bQ} {{\calq}}   
\renewcommand {\bG} {{\calr}}   

\begin{document}

\title{McCool groups of  toral relatively hyperbolic groups}
\author{Vincent Guirardel and Gilbert Levitt}
\date{\today}

\maketitle

\begin{abstract}
The outer automorphism group $\Out(G)$ of a group $G$ acts on the set of conjugacy classes of elements of $G$.
McCool proved that the stabilizer $\M(\calc)$ of a finite set of conjugacy classes is finitely presented when $G$ is free. More generally, we consider the group $\M(\calh)$ of outer automorphisms $\Phi$ of $G$ acting trivially on a family of subgroups $H_i$, in the sense that  $\Phi$ has representatives $\alpha_i$ with $\alpha_i$ equal to the identity on $H_i$.

When $G$ is a toral relatively hyperbolic group, we show that these two  definitions lead to the same subgroups of $\Out(G)$, which 
we call   ``McCool groups'' of G.
We prove that such McCool groups are  of type VF (some finite index subgroup has a finite classifying space).  Being of type VF also holds for the group of automorphisms 
of $G$   preserving a splitting of $G$  over abelian groups. 

We   show that McCool groups satisfy a uniform chain condition: there is a bound, depending only on $G$, for the length of a strictly decreasing sequence of McCool groups of $G$. Similarly, fixed subgroups of automorphisms of $G$ satisfy a uniform chain condition.
\end{abstract}

\section{Introduction}

Mapping class groups of punctured surfaces  may be viewed as subgroups of $\Out(F_n)$ for some $n$ (with $F_n$ denoting the free group of rank $n$). Indeed, they consist of automorphisms of $F_n$ fixing conjugacy classes corresponding to punctures. More generally, the group of automorphisms of $F_n$ fixing a finite number of conjugacy classes was studied by McCool, who proved in particular that such groups are finitely presented \cite{McCool_fp}. 
We therefore define:

\begin{dfn}\label{mcc1}
Let $G$ be a group. Let $\calc$ be a   set of conjugacy classes $[c_i]$ of elements of $G$. 
We denote by  $\M(\calc)$   the subgroup of $\Out(G)$ consisting of outer automorphisms fixing each $[c_i]$.  
If $\calc$ is finite, we say that $\M(\calc)$ is an   \emph{elementary McCool group} of $G$ (or of  $\Out(G)$).
\end{dfn}

Work on automorphisms suggests a more general definition: 

\begin{dfn}\label{gmc} 
Let $G$ be a group. Let $\calh=\{H_i \}$ be an arbitrary     family of  subgroups of $G$. 
We say that $\varphi\in\Aut(G)$, and its image $\Phi\in\Out(G)$, \emph{act trivially on $\calh$} if $\varphi$ acts on each $H_i$ as conjugation by some  $g_i\in G$.   Note that $\Phi $ acts trivially  if and only if it has representatives $\varphi_i\in\Aut(G)$ with $\varphi_i$ equal to the identity on $H_i$.

We denote by $\M(\calh)$, or $\M_G(\calh)$, the subgroup of $\Out(G)$ consisting of all $\Phi$ acting trivially on $\calh$.

If $\calh$ is a finite  family of finitely generated subgroups, 
we say that  $\M(\calh)$ 
is a \emph{McCool group} of $ G$ (or of  $\Out(G)$).
\end{dfn}

Elementary McCool groups correspond to McCool groups  with $\calh$  a finite family of cyclic groups.  $\M(\calh)$ does not change if we replace 
  the $H_i$'s by conjugate subgroups, so it is really associated to a family of conjugacy classes of subgroups.

For a topological analogy, 
one may think of $\M(\calh)$ as the group of automorphisms of $G=\pi_1(X)$ induced by homeomorphisms of $X$ equal to the identity on   subspaces $Y_i$ with $\pi_1(Y_i)=H_i$.

McCool groups are   relevant for automorphisms for the following reason (see \cite{GL6}). 
Consider  a splitting of a group $\hat G$ as a graph of groups in which $G$ is a vertex group, and 
the $H_i$'s  
are
the incident edge groups.  
Then   any element of $\M_G(\calh)$
 extends ``by the identity'' to an automorphism of $\hat G$. 
 Topologically, if $X$ is a vertex space in a graph of spaces $\hat X$, and edge spaces are attached to subspaces $Y_i\inc X$, then 
 any homeomorphism of $X$ equal to the identity on the $Y_i$'s   extends to $\hat X$ by the identity.
 
In this paper we will consider McCool groups when $G$ is a \emph{toral relatively hyperbolic group}: 
$G$ is torsion-free, and hyperbolic relative to a finite set of finitely generated abelian subgroups. This includes in particular torsion-free hyperbolic groups,   limit groups, and groups acting freely on $\R^n$-trees. 

We will show (Corollary \ref{genmc})  that in this case 
any $\M(\calh)$ is an elementary McCool group $\M(\calc)$; in other words, 
it is equivalent for a subgroup of $\Out(G)$ to be 
an elementary McCool group  $\M(\calc)$, or to be  a McCool group   $\M(\calh)$ with $\calh$ a finite family of finitely generated groups, or to be   $\M(\calh)$ with $\calh$ arbitrary.
We will not always make the distinction in the statements given below.

 It was proved by McCool \cite{McCool_fp} that   (elementary) McCool groups of a free group are finitely presented. 
Culler-Vogtmann \cite[Corollary 6.1.4]{CuVo_moduli} proved that they are of \emph{type VF}: they have a finite index subgroup  with a finite  classifying space (i.e.\ there exists a classifying space which  is a finite complex).  We   proved in \cite{GL6} that $\Out(G)$ is of type VF if $G$ is toral relatively hyperbolic  (in particular,  $\Out(G)$ is virtually torsion-free).  Our first main results extend this to certain naturally defined subgroups of $\Out(G)$.

\begin{thm} \label{mc}
If $G$ is a toral relatively hyperbolic group, 
then any   McCool group $\M(\calh)\inc\Out(G)$ is of type VF.
\end{thm}

\begin{thm} \label{mct}
If $G$ is a toral relatively hyperbolic group, and $T$ is a simplicial tree on which $G$ acts with abelian edge stabilizers, then the group of automorphisms $\Out(T)\inc\Out(G)$    leaving $T$ invariant  is of type VF.
\end{thm}

Our most general result in this direction (Corollary \ref{thm_general}) combines these two theorems; it   implies in particular that  \emph{$\M(\calh)\cap\Out(T)$ is of type VF if $T$ is as above and $\calh$ is any family of subgroups, each of which  fixes a point in $T$.}

\begin{rem*}  
Some of these results may be extended to groups which are hyperbolic relative to virtually polycyclic subgroups, but with the weaker conclusion that the automorphism groups   are of type $F_\infty$ (see \cite{GL_extension}).   On the other hand,   one can show that, if there exists a hyperbolic group which is not residually finite, 
 then there exists a hyperbolic group with $\Out(G)$ not virtually torsion-free (hence not VF).
\end{rem*}
Our second main result is the following:

\begin{thm} \label{mccc}
Let $G$ be a toral relatively hyperbolic group. 
  McCool groups  of $G$ satisfy a \emph{uniform chain condition:}
 there exists $C=C(G)$ such that, if 
 $$\M(\calh_0)\incd \M(\calh_1)\incd\dots\incd\M(\calh_p ) $$ 
is a strictly decreasing chain of   McCool groups  in $\Out(G)$, then $p\le C$. 
\end{thm}

This is based, among other things, on the vertex finiteness proved in  \cite{GL_vertex}:
if $G$ is toral relatively hyperbolic, then all vertex groups occurring in splittings of $G$ over abelian groups
lie in finitely many isomorphism classes. 

The chain condition, proved in Section \ref{pfcc} for McCool groups $\M(\calh)$ with $\calh$ a finite family of finitely generated groups,  implies:
\begin{cor}\label{genmc}
Let $G$ be a toral relatively hyperbolic group. 
If  $\calh $ is a (possibly infinite) family of (possibly infinitely generated) subgroups $H_i\inc G$,
there exists a finite set of conjugacy classes   $\calc$ 
such that $\M(\calh)=\M(\calc)$. In particular, any $\M(\calh)$ is a McCool group, and any McCool group  is an elementary McCool group $\M(\calc)$. 
\end{cor}

The chain condition also implies that 
no McCool group  $\M(\calh)\inc\Out(G)$ is conjugate to a proper subgroup. Note, however, that McCool groups may fail to be co-Hopfian (they may be isomorphic to proper subgroups). To illustrate the variety of McCool groups, we show:

\begin{prop} \label{infmc}
$\Out(F_n)$ contains infinitely many non-isomorphic McCool groups if $n\ge4$; it  contains infinitely many non-conjugate McCool groups  if $n\ge3$.
\end{prop}

It may be shown that the bounds on $n$ are sharp (see the appendix). 
 We will also show in the appendix that, if  $G$ is a  torsion-free  \emph{one-ended} hyperbolic group, then $\Out(G)$ only contains finitely many McCool groups up to conjugacy.

Say that $J\inc G$ is   \emph{a fixed subgroup} if there is a family of automorphisms $\alpha_i\in\Aut(G)$ such that $J=\cap_i\Fix\alpha_i$,  with $\Fix\alpha=\{g\in G \mid \alpha(g)=g\}$. The chain condition also implies:

\begin{thm}\label{uccfix} 
Let $G$ be a toral relatively hyperbolic group. There is a constant $c=c(G)$ such that, if $J_0\incs J_1\incs\dots\incs J_p$ is a strictly ascending chain of fixed subgroups, then $p\le c$. 
\end{thm}

This was proved by Martino-Ventura  
\cite{MaVe_fixed}
for $G$ free, with $c(F_n)=2n$.  In \cite{GL7}, we will apply Theorems \ref{mc} and  \ref{uccfix} to the study of stabilizers for the action of $\Out(G)$ on spaces of \Rt s.

As explained above, one does not get new groups by allowing the set $\calc$ in Definition \ref{mcc1} to be infinite, or by considering arbitrary subgroups as in Definition \ref{gmc}. 
  The following definition provides a genuine generalization. 

\begin{dfn}
Let $G$ be a group, and $\calc$   a finite set of conjugacy classes  $[c_i]$. We write $\calc\m$  for the set of classes $[c_i\m]$. Let $\wh\M(\calc)$ be the subgroup of $\Out(G)$ consisting of automorphisms leaving $\calc\cup\calc\m$ globally invariant; it contains $\M(\calc)$ as a normal subgroup of finite index.
We say that $\wh\M(\calc)$ is an \emph{extended elementary McCool group} of $G$. 
\end{dfn}

More generally, if $\calh $ is a finite family of  subgroups, 
one can  define 
finite extensions of $\M(\calh)$ by allowing the $H_i$'s to be permuted, or the action on $H_i$ to be only ``almost'' trivial.

\begin{prop} \label{mccet}
Given  a toral relatively hyperbolic group $G$, there exists a number $C$ such that,   if a subgroup  $\wh M\subset \Out(G)$  contains a   group $  \M(\calh)$ with finite index, then 
 the index $[\wh M:\M(\calh)]$ is bounded by $C$.

In particular, for $\calc$ finite, the index of $\M(\calc)$ in $\wh\M(\calc)$ is bounded by a constant depending only on $G$.
\end{prop}

It follows that extended elementary McCool groups satisfy a uniform chain condition as in Theorem \ref{mccc} (see Corollary \ref{eucc}). We also have:

\begin{cor} \label{rless}
Let $G$ be a toral relatively hyperbolic group.   Let $A$ be any subgroup of $\Out(G)$, and let $\calc_A$ be the (possibly infinite) set of  conjugacy classes of $G$ whose $A$-orbit is finite. The image of $A$ in the group of permutations of $\calc_A$ is finite, and its order is bounded by a constant depending only on $G$. In other words, there is a subgroup $A_0\inc A$ of bounded finite index  such that every conjugacy class in $G$  is fixed by $A_0$ or has infinite orbit under $A_0$. 
\end{cor}

When $G$ is free, one may take for $A_0$ the intersection of $A$ with a fixed finite index subgroup of $\Out(G)$  (independent of $A$) \cite{HaMo_announcement}.

 One may also consider subgroups of $\Aut(G)$. 
 \begin{dfn}
Let $\calh$ be a family of (conjugacy classes of) subgroups, and $H_0<G$ another subgroup.
Let   $ 
\AM(\calh, H_0)\inc\Aut(G)$ be the group of automorphisms acting trivially on $\calh$ (in the sense of Definition \ref{gmc}) and fixing the elements of $H_0$.
 \end{dfn}

\begin{prop} \label{mcaut} 
If $G$ is a non-abelian toral relatively hyperbolic group, 
then  $ 
 \AM(\calh, H_0)$ is an extension $$1\to K\to
 \AM(\calh, H_0)\to \M(\calh')\to 1$$ where $\M(\calh')\inc\Out(G)$ is a McCool group, and   $K $ is the centralizer of $H_0$ (isomorphic to $G$ or to $\Z^n$ for some $n\ge0$).
\end{prop}

\begin{cor} \label{mca}
Theorems \ref{mc} and   \ref{mccc} 
also hold  in $\Aut(G)$: groups of the form $ 
 \AM(\calh, H_0)$ are of type VF 
and satisfy a uniform chain condition.
\end{cor}

  Theorems \ref{mc}  and \ref{mct} are proved in Section \ref{clas},  and  Theorem \ref{mccc}  is proved   in Section \ref{pfcc}.  All other results are proved in Section  \ref{pfcor}.

\paragraph{Acknowledgements}  
The first author acknowledges support from ANR-11-BS01-013, the Institut Universitaire de France, and from the Lebesgue center of mathematics. 
The second author acknowledges support from  ANR-10-BLAN-116-03.

\section{Preliminaries}

In this paper, $G$ will always denote a toral  
relatively hyperbolic group. 
Any non-trivial abelian subgroup $A$ of $G$ is contained in  a unique maximal abelian subgroup.
 The maximal abelian subgroups   are malnormal ($G$ is CSA), finitely generated, and there are finitely many non-cyclic ones up to conjugacy. Two subgroups of $A$ which are conjugate in $G$ are equal.

  The center of a group $H$ will be denoted by $Z(H)$.  We write $N_K(H)$ for the normalizer of a group $H$ in a group $K$, with $N(H)=N_G(H)$. Centralizers are denoted by $Z_K(H)$.
  
We say that $\Phi\in\Out(G)$ preserves a subgroup $H$, or leaves $H$ invariant, if its representatives $\varphi\in\Aut(G)$ map $H$ to a conjugate. If $\varphi\in\Aut(G)$ equals the identity on $H$, we say that it fixes $H$.

\begin{dfn} \label{pres}
 If $\calh$ is a family of subgroups, we let $ \Out(G;\calh)\inc\Out(G)$ be the group of automorphisms preserving each $H\in \calh$, and $\widehat\Out(G;\calh)$   the group of automorphisms preserving $\calh$ globally (possibly permuting  groups in $\calh$).

We denote by  $$\Out(G;\mk\calh)=\M(\calh)\inc \Out(G)$$ the group of automorphisms acting trivially on groups in $\calh$ (as in Definition \ref{gmc}). 

 We write $$\Out(G;\mk\calh,\calk):=\Out(G;\mk\calh )\cap\Out(G; \calk),$$ and $$\Out(G;\calh,\calk):=\Out(G;\calh \cup\calk). $$
\end{dfn}

\begin{rem*} 
 $\Out(G;\mk\calh)$ and $\M(\calh)$ denote the same group. The notation $\Out(G;\mk\calh)$ is more flexible and will be   convenient in Section \ref{clas}.

We will often view a set of conjugacy classes $\calc=\{[c_i]\}$ as a family of cyclic subgroups $\calh=\{\langle c_i\rangle\}$, 
since $\M(\calc)=\M(\calh)$.  Note 
that $\Out(G;\calh)$ is larger than $\M(\calc)=\M(\calh)$
 since $c_i$ may sent
to a conjugate of $c_i\m$. 
\end{rem*}

 For example, suppose that  $H<G=\bbZ^n$ is the subgroup generated by the first $k$ basis elements, and $\calh=\{H\}$. Then $\Out(G)=GL(n,\Z)$; the group  $\Out(G;\calh)$ consists of  block triangular matrices, 
 and $\Out(G;\mk\calh)=\M(\calh)$ is the group of matrices fixing the first $k$ basis vectors.

There are   inclusions $\Out(G;\mk\calh)\subset \Out(G;\calh)\subset \wh\Out(G;\calh)$.
 Note that $\Out(G;\mk\calh)$ has finite index in $ \Out(G;\calh)$ and $\widehat\Out(G;\calh)$ if $\calh$ is a finite family of cyclic groups.

Given a family $\calh$ and a subgroup $J$, we denote by $\calh_{ | J}$ the $J$-conjugacy classes of subgroups of $J$ conjugate to a group of    $\calh$.  We view $\calh_{ | J}$ as   a family of subgroups of $J$, each defined up to conjugacy in $J$.  In the next subsection 
we will define a closely related notion $\calh_{||J}$ when $J=G_v$ is a vertex stabilizer in a tree.

If $\calc$ is a set of conjugacy classes $[c_i]$,   viewed as a set of cyclic subgroups, 
   $\calc_{ | J}$ is the set of   $J$-conjugacy classes of elements of $J$ representing elements in $\calc$.

  Now suppose that
  subgroups of $J$ which are conjugate in $G$ are conjugate in $J$; 
  this   holds for instance if $J$ is malnormal (in particular if $J$ is a free factor), and also if $J$ is   abelian. 
 In this case
 we may view $\calh_{ | J}  $ as a subset of $\calh$; it is finite if $\calh$ is.

\subsection{Trees and splittings} \label{tre}
A tree will be a simplicial tree $T$ with an action of $G$ without inversions. 
 A tree $T$ is \emph{relative to $\calh$} (resp.\ to $\calc$) if any group  in $\calh$ (resp.\ any element  representing a class in $\calc$) fixes a point in $T$.

Two trees are considered to be the same if there is a $G$-equivariant isomorphism between them. In this paper, all trees will have abelian edge stabilizers.

Unless mentioned otherwise, we
 assume  that the action is \emph{minimal} (there is no proper invariant subtree). We usually assume    that there is \emph{no redundant vertex} (if $T\setminus \{x\}$ has two components, some $g\in G$ interchanges them). If a finitely generated subgroup $H\inc G$ acts on $T$ with no global fixed point, there is a smallest $H$-invariant subtree called the \emph{minimal subtree} of $H$.

 The tree $T$ is \emph{trivial} if there is a global fixed point (minimality then implies that $T$ is a point).  
An element of $G$, or a subgroup, is \emph{elliptic} if it fixes a point in $T$. Conjugates of elliptic subgroups are elliptic, so we also consider elliptic conjugacy classes.

An action of $G$ on  a tree $T$ gives rise to a splitting of $G$, i.e.\ a decomposition of $G$ as the fundamental group of  the quotient graph of groups $\Gamma=T/G$. Conversely, $T$ is the Bass-Serre tree of $\Gamma$. 
All definitions given here apply to both splittings and trees.   In particular, a splitting is relative to $\calh$ if every $H\in\calh$ has a conjugate contained in  a vertex group.

 Minimality implies that the graph  $\Gamma$ is finite. There is a one-to-one correspondence between vertices (resp.\ edges) of $\Gamma$ and $G$-orbits of vertices (resp.\ edges) of $T$. 
We denote by $V$ the set of vertices of $\Gamma$, and by $G_v$ the group carried by  a vertex $v\in V$. We also view $v$ as a vertex of $T$ with stabilizer $G_v$. Similarly, we denote by $e$ an edge of $\Gamma$ or $T$,   by $G_e$ the corresponding group (always abelian in this paper), and by $E$ the set of non-oriented edges of $\Gamma$. 

Edge groups   being abelian, hence relatively quasiconvex, every vertex group $G_v$ is toral relatively hyperbolic (see for instance \cite{GL6}). 

The edge groups carried by edges of $\Gamma$ incident to a given vertex $v$ will  be called the \emph{incident edge groups}  of $G_v$. We denote by $\Inc_v$ the family of incident edge groups (we view it as a finite family of subgroups of $G_v$, each well-defined up to conjugacy).

 If $\calh$ is a finite family of subgroups of $G$, and $v$ is a vertex stabilizer of $T$, we denote by   $\calh_{||G_v}$ the family of subgroups $H\inc G_v$ which are conjugate to a group of $\calh$ and fix no other point in $T$. Two such groups are conjugate in $G_v$ if they are conjugate in $G$ (\cite{GL6}, Lemma 2.2 where the notation $\calh_{|G_v}$ is used instead), 
so we may also view $\calh_{||G_v}$ as a subset of $\calh$ (it contains some of the  groups of $\calh$ having a conjugate in $G_v$), or as a finite family of subgroups of $G_v$, 
each well-defined up to conjugacy 
($\calh_{||G_v}$ may be smaller than   $\calh_{ | G_v}$ because we do  not include subgroups of edge groups).
 
Any splitting of $G_v$ relative to $\Inc_v $ extends to a splitting of $G$.  If   $T$ is relative to $\calh$, any splitting of $G_v$ relative to $\Inc_v\cup \calh_{||G_v}$  is relative to  $\calh_{|G_v}$ and extends to a splitting of $G$ relative to $\calh$.

 If  $\calc$ is a   set of conjugacy classes, we view 
 $\calc_{  | | G_v}$ 
as the subset of $\calc$ consisting of classes having a representative that fixes $v$ and no other vertex. 
In particular, $\calc_{  | | G_v}$ is finite if $\calc$ is.

A tree $T'$ is a \emph{collapse} of $T$ if it is obtained from $T$ by collapsing 
each edge in a certain $G$-invariant
collection 
to a point; conversely, we say that $T$ \emph{refines} $T'$. 
In terms of graphs of groups, one passes from $\Gamma=T/G$ to $\Gamma'=T'/G$ by collapsing edges;
  for each vertex $v'\in\Gamma'$, 
the vertex group $G_{v'}$ is the fundamental group of the graph of groups $\Gamma_{v'}$ occuring as the preimage of $v'$ in $\Gamma$.

All maps between trees will be $G$-equivariant. 
 Given two trees $T$ and $T'$, we say that $T$ \emph{dominates} $T'$ if there is a 
 map $f:T\to T'$, or equivalently if every subgroup which is elliptic in $T$ is also elliptic in $T'$; in particular, $T$ dominates any collapse $T'$.  We sometimes say that $f$ is a \emph{domination map.} 
Minimality implies that it is onto.

Two trees  belong to the same \emph{deformation space} if they dominate each other. In other words, a deformation space $\cald$ is   the set of all trees having a given family of subgroups as their elliptic subgroups. We say that $\cald$ dominates $\cald'$ if trees in $\cald$ dominate those in $\cald'$.

\subsection{JSJ decompositions \cite{GL3a,GL3b}}\label{jsj}

Let $\calh$ be a family of subgroups of $G$.  
  Recall that a tree $T$ is \emph{relative} to $\calh$ 
  if all groups of $\calh$ 
  are elliptic in $T$. 

We denote   
by $\hp$ the family obtained by adding to $\calh$ all non-cyclic abelian subgroups of $G$.

The group $G$ is \emph{freely indecomposable} relative to $\calh$ if it does not split over the trivial group relative to $\calh$; equivalently, 
$G$ cannot be written non-trivially as $A*B$ with every group of $\calh$ contained in a conjugate of $A$ or $B$  (if $\calh$ is trivial, we also require $G\ne\Z$, as  we consider  $\Z$ as freely decomposable). 
Non-cyclic abelian groups being one-ended, being {freely indecomposable} relative to $\calh$
is the same as relative to $\hp$.

Let $\cala$ be another family of subgroups (in this paper, $\cala$ 
consists of the trivial group or is the family of all abelian subgroups). 
Once $\calh $ and $\cala$ are fixed, we only consider trees relative to $\calh$, with edge stabilizers in $\cala$. We also assume that trees are minimal.

 A tree $T$ (with edge stabilizers in $\cala$, relative to $\calh$) is \emph{universally elliptic}  (with respect to $\calh$) if its edge stabilizers are elliptic in every tree. It is a \emph{JSJ tree} if,  moreover,  it dominates every universally elliptic tree.
   The set of JSJ trees is   called the \emph{JSJ deformation space} (over $\cala$ relative to $\calh$). All JSJ trees have the same vertex stabilizers, provided one restricts to stabilizers not in $\cala$.
  
  When $\cala$ consists of the trivial group, the JSJ deformation space is called  the \emph{Grushko deformation space} (relative to $\calh$). The group $G$ has a relative Grushko decomposition $G=G_1*\dots*G_n*F_p$, with $F_p$ free,  every $H\in\calh$ contained in some $G_i$ (up to conjugacy), and $G_i$ freely indecomposable relative to   $\calh_{ | G_i}$.  
    Vertex stabilizers of the relative Grushko deformation space $\cald$ are precisely conjugates of the $G_i$'s. 
    The   deformation space 
 is trivial (it only contains the trivial tree) if and only if $G$ is freely indecomposable  relative to $\calh$. 
  Writing $\calg=\{G_1,\dots, G_n\}$, note that $\Out(G;\calh\cup\calg)$ has finite index in $\Out(G;\calh )$, because automorphisms in $\Out(G;\calh )$ leave $\cald$ invariant and therefore permute the $G_i$'s (up to conjugacy).

Now suppose that $\cala$ consists of all abelian subgroups, and $G$ is freely indecomposable relative to a family $\calh$. Then \cite[11.1]{GL3b} the 
  JSJ deformation space relative to $\hp$ 
  contains a preferred tree $\Tcan$; this tree  is  invariant under  $\widehat\Out(G;\calh)$ (the group of automorphisms preserving $\calh$). 
  
  It is obtained as a \emph{tree of cylinders}. We   describe this construction in the case that will be needed here (see Proposition 6.3 of \cite{GL4} for details). 
  Let $T$ be any tree with non-trivial abelian edge stabilizers, relative to all non-cyclic abelian subgroups. Say that two edges $e,e'$ belong to the same cylinder if their stabilizers commute. Cylinders are subtrees intersecting in at most one point. 
  
  The tree of cylinders $T_c$ is defined as follows. It is bipartite, with vertex set $\V_0\cup \V_1$. Vertices in $\V_0$ are vertices of $T$ belonging to at least two cylinders. Vertices in $\V_1$ are cylinders of $T$.  A vertex $v\in \V_0$ is joined to a vertex 
$ Y\in \V_1$ if $v$ (viewed as a vertex of $T$) belongs to $Y$ (viewed as a subtree of $T$). Equivalently, one obtains $T_c$ from $T$ by replacing each cylinder $Y$ by the cone on its boundary (points of $Y$ belonging to at least one other cylinder).

The tree $T_c$    only depends on the 
deformation space $\cald$ containing  $T$, and it belongs to $\cald$. Like $T$, it has non-trivial abelian edge stabilizers, and is relative to all non-cyclic abelian subgroups. 
It is minimal if $T$ is minimal, but vertices in $\V_1$ may be redundant vertices.

The stabilizer of a vertex   $v_1\in \V_1 $ is a maximal abelian subgroup.  
 The stabilizer  of a vertex in $\V_0$ is non-abelian and is the stabilizer  of a vertex of $T$. The  stabilizer of an edge $v_0v_1$ with $v_i\in \V _i$ is  an infinite abelian subgroup, it is a maximal abelian subgroup of $G_{v_0}$ (but it is not always maximal  abelian in  $G_{v_1}$).

 The $\widehat\Out(G;\calh)$-invariant tree $\Tcan$ mentioned above is the tree of cylinders of JSJ trees relative to $\hp$.  It is a JSJ tree, and 
the tree of cylinders of $\Tcan$ is $\Tcan$   itself.

 Let $\Gcan=\Tcan/G$ be the quotient graph of groups, and let $v\in \V_0/G$ be a vertex with $G_v$ non-abelian. 
 If $G_v$ does not split over an abelian group relative to incident edge groups and 
to  $\calh_{ | |  G_v}$, it is universally elliptic (with respect to both $\calh$ and $\hp$); we say that $G_v$ (or $v$) is \emph{rigid}. Otherwise it is \emph{flexible}. 
 
 A key fact here is that every flexible vertex $v$ of $\Gcan$ is \emph{quadratically hanging (QH)}. 
The group  $G_v$ is the fundamental group of a compact (possibly non-orientable) surface $\Sigma$, and incident edge groups are boundary subgroups of $\pi_1(\Sigma)$ (i.e.\ fundamental groups of boundary components of $\Sigma$); in particular, incident edge groups are cyclic. At most one incident edge group is attached to a given boundary component (groups carried by distinct incident edges are non-conjugate in $G_v$).  If $H$ is conjugate to a group of $\calh$, then $H\cap G_v$ is contained in a boundary subgroup. Conversely,  every boundary subgroup is an incident edge group or has a finite index subgroup which is conjugate to a group of $\calh$.  

As in \cite{Szepietowski_presentation}, 
we   denote   by $\calp\calm^+(\Sigma)$ the group of isotopy classes of homeomorphisms of   $\Sigma$ mapping each boundary component to itself in an orientation-preserving way.  We view $\calp\calm^+(\Sigma)$ as a subgroup of $\Out(\pi_1(\Sigma))=\Out(G_v)$, indeed $\calp\calm^+(\Sigma)=\Out(G_v;\mk\Inc_v ,\mk\calh_{||G_v} )$.

\subsection{Automorphisms of trees}\label{automs}

There is a natural action of 
$\Out(G)$ on the set of trees, given by precomposing the action on $T$ with an automorphism of $G$. We denote by $\Out(T)$ the stabilizer of a tree $T$.  We write $\Out(T,\calh)$ for $\Out(T)\cap\Out(G;\calh)$, and so on.

If $T$ is a point, $\Out(T)=\Out(G)$. If $G$ is abelian and $T$ is not a point, then $T$ is a line on which $G$ acts by 
 integral translations, and $\Out(T)$ is the group of automorphisms  of $G$ preserving the kernel of the action. 

We now study $\Out(T)$ in the general case, following \cite{Lev_automorphisms}. 

We  always assume that edge stabilizers are abelian.  
This implies that all vertex or edge stabilizers  $H$    have the property that the normalizer $N(H)$ acts on $H$ by inner automorphisms: indeed, $N(H)$ is abelian if $H$ is abelian,  equal to $H$   if $H$ is not abelian.

One first considers the action of $\Out(T)$ on the finite graph $\Gamma=T/G$. We always denote by $\Out^0(T)$ the finite index subgroup consisting of automorphisms acting trivially.

We study it through the natural map  $$\rho=\prod_{v\in V}\rho_v:\Out^0(T)\to\prod_{v\in V}\Out(G_v)$$ recording the action of automorphisms on vertex groups (see Section 2 of \cite{Lev_automorphisms}); recall that $V$ is the vertex set of $\Gamma$. 
Since $N(G_v)$ acts on $G_v$ by inner automorphisms,  
  $\rho_v(\Phi)$ is simply defined as the class of $\alpha_{ | G_v}$, where $\alpha\in\Aut(G)$ is any representative of  $\Phi\in\Out^0(T)$   leaving $G_v$ invariant.

The image of $\rho$ is contained in $\prod_{v\in V} \Out(G_v;\Inc_v)$  (the family of incident edge groups  at a given $v$ is preserved). It  contains the subgroup $\prod_{v\in V}\Out(G_v;\mk\Inc_v)$ 
because automorphisms of $G_v$ acting trivially on incident edge groups extend ``by the identity'' to automorphisms of $G$ preserving $T$.  

The kernel of $\rho$ is the \emph{group of twists}  $\calt$, a finitely generated abelian group  when no edge group  is trivial
(bitwists as defined in  \cite{Lev_automorphisms} belong to $\calt$ because the normalizer of an abelian subgroup is its centralizer). We therefore have an exact sequence
$$1\to\calt\to \Out^0(T)\xra{\ \rho\ }\prod_{v\in V} \Out(G_v;\Inc_v).$$

Now suppose that $T$ is relative to   families $\calh$ and $\calk$  (\ie   each $H_i$, $K_j$ fixes a point in $T$). 
A trivial but important remark is that $\calt\inc\Out(G ;\mk \calh, \mk\calk )$. As pointed out in Lemma 2.10 of \cite{GL6}, we have
$$\prod_{v\in V}\Out(G_v;\mk\Inc_v ,\mk\calh_{||G_v},  \calk_{||G_v})  \inc \rho\biggl (\Out^0(T)\cap\Out(G;\mk\calh, \calk)\biggr)
\inc \prod_{v\in V} \Out(G_v;\Inc_v,\mk\calh_{||G_v}, \calk_{||G_v})
$$
(see Subsection \ref{tre} for the definition of $\calh_{||G_v}$; groups of $\calh_{||G_v}$ that are conjugate
in $G$ are necessarily conjugate in $G_v$).

 The fact noted above that the image of $\Out^0(T)$ by $\rho$ contains $\prod_{v\in V} \Out(G_v;\mk\Inc_v)$ expresses that  \emph{automorphisms $\Phi_v\in\Out(G_v)$ acting trivially on incident edge groups may be combined into a global $\Phi\in\Out(G)$}. In Subsection \ref{finpf} we will need a more general result, where we only assume that the $\Phi_v$'s have compatible actions on edge groups.
 
 Given  an edge $e$ of $\Gamma$,
  there is a natural map $\rho_e:\Out^0(T)\to \Out(G_e)$,  defined in the same way as $\rho_v$ above. If $v$ is an endpoint of $e$, the inclusion of $G_e$ into $G_v$ induces a homomorphism ${\rho_{v,e}}:\Out(G_v;\Inc_v)\to\Out(G_e)$ with $\rho_e={\rho_{v,e}}\circ\rho_v$
(it is well-defined because the  normalizer $N_{G_v}(G_e)$ acts on $G_e$   by inner automorphisms).

\begin{lem} \label{modif}
 Consider a family of automorphisms $\Phi_v\in\Out(G_v;\Inc_v)$  such that, if  $e=vw$ is any edge of $\Gamma$, then $\rho_{v,e}(\Phi_v)=\rho_{w,e}(\Phi_w)$. There exists $\Phi\in\Out^0(T)$ such that $\rho_v(\Phi)=\Phi_v$ for every $v$.
\end{lem}
 
 We leave the proof to the reader.
  The lemma applies 
  to any 
   graph of groups such that, for every vertex or edge group $H$, the  normalizer $N(H)$ acts on $H$   by inner automorphisms. 
$\Phi$ is not unique, it may be composed with any element of $\calt$.

In Subsection \ref{finpf} we will have a family of automorphisms $\Phi_e\in\Out(G_e)$, and we will want  $\Phi\in\Out^0(T)$ such that $\rho_e(\Phi)=\Phi_e$ for every $e$. By the lemma, it suffices to find automorphisms $\Phi_v\in\Out(G_v;\Inc_v)$ inducing the $\Phi_e$'s.
 
\subsection{Rigid vertices}
  
  We now specialize to the case when 
  $T=\Tcan$ is the canonical JSJ decomposition relative to $\hp$ discussed in 
     Subsection \ref{jsj}.
 
 If $v$ is a QH vertex, the image of $\Out^0(T)\cap\Out(G;\mk\calh)$ in $\Out(G_v)$ contains 
 $\calp\calm^+(\Sigma)=\Out(G_v;\mk\Inc_v ,\mk\calh_{||G_v} )$ with finite index (see \cite{GL6},   Proposition 4.7).

If $v$ is a rigid vertex, then $G_v$ does not split over an abelian group relative to $\Inc_v\cup\calh_{||G_v}$.
By the Bestvina-Paulin method and Rips theory, one deduces that the image of $\Out^0(T)\cap\Out(G;\mk\calh)$ in $\Out(G_v)$ is finite if $\calh$ is a finite family of finitely generated subgroups (see \cite{GL6}, Theorem 3.9 and Proposition 4.7).

\begin{lem}
\label{lem_fini}
   Let $\calh,\calk$ be   finite families of 
  finitely generated subgroups, with each group in $\calk$ abelian.
 Assume that $G$ is one-ended relative to $\calh\cup\calk$,
 and let $\Tcan$ be the canonical JSJ tree relative to $(\calh\cup\calk)^{+ab}$.
 
The image of  $$\Out^0(T)\cap\Out(G;\mk\calh,\calk)$$ by  $\rho_v:\Out^0(T)\to\Out(G_v)$ is finite if $v$ is a rigid vertex of $\Tcan$.    Its image 
 by $\rho_e:\Out^0(T)\to\Out(G_e)$ is finite if $e$ is any edge.
\end{lem}

\begin{proof}
 Define $\calk_\Z$ by removing  all non-cyclic groups from $\calk$.
Being freely indecomposable relative to $\calh\cup\calk$ is the same as being freely indecomposable relative to $\calh\cup\calk_\Z$, and   a tree is relative to $ (\calh\cup\calk)^{+ab}$ if and only if it is relative to $(\calh\cup\calk_\Z)^{+ab}$.  We may  therefore  view $\Tcan$ as the canonical JSJ tree relative to $(\calh \cup\calk_\Z)^{+ab}$.

Let $v$ be a rigid vertex. 
The group $\Out(G;\mk\calh,\calk)$ is contained in $\Out(G;\mk\calh,\calk_\Z)$, which contains
$\Out(G;\mk\calh,\mk\calk_\Z)$
with finite index.
 As explained above, 
the image of $\Out^0(T)\cap\Out(G;\mk\calh,\mk\calk_\Z)$ in $\Out(G_v)$ is finite (\cite{GL6}, Prop.\ 4.7). 
The first assertion of the lemma follows.

 Since $\Tcan$ is bipartite, every edge $e$ is incident to a vertex $v$ which is QH   or    rigid.
In the first case $G_e$ is cyclic, so there is nothing to prove. In the second case the map $\rho_e:\Out^0(T)\to\Out(G_e)$ factors through $\Out(G_v)$, and the second assertion follows from the first.
\end{proof}

\section{Finite classifying space}\label{clas}

In this section, we   prove that  
McCool groups of a toral relatively hyperbolic group have type VF (Theorem \ref{mc}),
and that so does the stabilizer of a splitting (Theorem \ref{mct}).
In the course of the proof,  we will  
describe the automorphisms of a given maximal abelian subgroup which are restrictions of an automorphism of $G$ belonging to a  given McCool group (Proposition \ref{image}).

We start  by recalling some standard facts 
about groups of type VF.

A group has type F if it has a finite classifying space, type VF if some finite index subgroup is of type F. A key tool for proving that groups have type F is the following  statement:

\begin{thm}[see for instance {\cite[Th.\ 7.3.4]{Geoghegan_topological}}] \label{extF}
Suppose that $G$ acts  simplicially and cocompactly   on a contractible simplicial complex $X$. If all point stabilizers have type F, so does $G$. 
In particular, being of type F is stable under extensions. \qed
\end{thm}

If $G$ has a finite index subgroup acting as in the theorem, then $G$ has type VF.  In particular:

\begin{cor}\label{cor_extensionVF}
  Given an exact sequence $1\ra N\ra G\ra Q \ra 1$,  suppose that   $Q$ has type VF, and  $G$ has a finite index   subgroup $G_0<G$
  such that $G_0\cap N$ has type F. Then $G$ has type VF. \qed
\end{cor}

\begin{rem} \label{subt}
Suppose that $G$ acts on $X$ as in Theorem \ref {extF}. If point   stabilizers are only of type VF,
one cannot claim that $G$ has type VF, 
even if $G$ is torsion-free. This subtlety was overlooked in Theorem 5.2 of \cite{GL1} (we will give a corrected statement in Corollary \ref{corr}), and it introduces technical complications
 (which would not occur if we only wanted to prove that the groups under consideration have  
 type $F_\infty$).
In particular,  to study the stabilizer of a tree with  non-cyclic edge stabilizers in Subsection \ref{sec_edge}, we have to prove     more precise   versions of certain  results (such as the  ``moreover'' in Theorem \ref{mcga}).
\end{rem}

\subsection{McCool groups are VF}

In this subsection we prove the following strengthening of Theorem \ref{mc}.

 \begin{thm} \label{mcga}
Let $G$ be a toral relatively hyperbolic group.
Let $\calh$ and $\calk$ be two finite  families of   finitely generated subgroups, 
      with each group in $\calk$   abelian.  
Then $\Out(G;\mk\calh,\calk) $ is of type VF. 

Morerover, if groups in $\calh$ are also abelian,  there exists a finite index subgroup \break  $\Out^1(G;\calh ,\calk) \inc\Out(G;\calh ,\calk) $ such that $\Out^1(G;\calh ,\calk) \cap \Out(G;\mk\calh,\calk) $ is of type $F$.
 \end{thm}

Recall  (Definition \ref{pres}) that  $\Out(G;\mk\calh ,\calk) $ consists of classes of automorphisms  acting trivially on each group $H_i\in  \calh$ (i.e.\ as conjugation by some $g_i\in G$), 
and leaving each $K_j\in\calk$ invariant up to conjugacy. 
 
 It will follow from 
  Corollary \ref{genmc}  
 that  the main assertion of Theorem \ref{mcga} holds 
 if $\calh$ is an arbitrary family of   subgroups  
 (see Corollary \ref{thm_general}), but finiteness is needed at this point  in order to apply Lemma \ref{lem_fini}.

\begin{con} \label{1}   In this section, a superscript  $-^1$,  as in $\Out^1(G;\calh ,\calk)$,   always indicates a subgroup of finite  index. The superscript  $-^0$ refers to a trivial action on a quotient graph of groups  (see Section \ref{automs}).
\end{con}

\subsubsection{The abelian case}
 
The following lemma deals with the case when $G=\Z^n$.

\begin{lem} \label{arithm}
 Let $\calh$ and $\calk$ be finite families of subgroups of $\Z^n$. Let $A=\Out(\bbZ^n;\mk\calh ,\calk) $ be the subgroup of $GL(n,\Z)$ consisting of matrices acting as the identity on  groups  $H_i\in\calh$ and leaving each $K_j\in\calk$ invariant. Then $A$ is of type VF. More precisely, every torsion-free subgroup of finite index $A'\inc A$ is of type F.
\end{lem}

Recall  that $GL(n,\Z)$ is virtually torsion-free, so groups such as $A'$ exist.

\begin{proof}
The set of  endomorphisms of $\Z^n$ acting as the identity on $H_i$ and preserving $K_j$ is a linear subspace defined by  linear equations with rational coefficients. 
It follows that the groups $A$ and $A'$ are arithmetic:   they are commensurable with a subgroup of $GL(n,\bbZ)$
defined by $\Q$-linear equations. 
By Borel-Serre \cite{BorelSerre_corners}, every torsion-free arithmetic subgroup of $GL(n,\Q)$ is of type F.
\end{proof}

To deduce Theorem \ref{mcga} when $G$ is abelian, we  simply define $\Out^1(G;\calh ,\calk)$ as any   torsion-free finite index  subgroup  of  $\Out(G;\calh ,\calk) $.

If $G$ is not abelian, we shall distinguish two cases.

\subsubsection{The one-ended case} \label{oe}

We   
first assume that $G$ is freely indecomposable relative to $\calh\cup\calk$:  
one cannot write $G=A*B$ with each group of $\calh\cup\calk$ contained in a conjugate of $A$ or $B$.  
We then consider the canonical tree $\Tcan$ as in Subsection \ref{jsj} (it is a JSJ tree relative to $\calh$, $\calk$, and   to non-cyclic abelian subgroups).  
 It is invariant under  
$\Out(G;\calh,\calk)$, so $\Out(G;  \calh,\calk)\inc\Out(\Tcan)$.

We write $\Out^0(\Tcan)$ for the finite index subgroup consisting of automorphisms acting trivially on   the finite graph $\Gcan=\Tcan/G$, 
and  $$\Out^0(G; \calh,\calk)=\Out (G; \calh,\calk)\cap\Out^0(\Tcan);$$ it has finite index in $\Out (G; \calh,\calk)$.

Recall that non-abelian vertex stabilizers $G_v$ of $\Tcan$ (or vertex groups of $\Gcan$) are  rigid  or QH. Also recall from Subsection \ref{automs} that, for each vertex $v$, there is a map $\rho_v:\Out^0(\Tcan)\to\Out(G_v; \Inc_v)$, with 
$\Inc_v$ the family of incident edge groups (see Subsection \ref{tre}).

We define a subgroup $\Out^r(G;\calh ,\calk)\inc\Out (G;\calh ,\calk)$ 
by restricting to automorphisms $\Phi\in\Out^0(G;\calh,\calk)$, 
and imposing  conditions on the image of $\Phi$ by the maps $\rho_v$.

$\bullet$ If $G_v$ is rigid,  we   ask that $\rho_v(\Phi)$ be trivial.

$\bullet$ If $G_v$ is abelian,  we fix a torsion-free subgroup  of finite index   $\Out^1(G_v)\inc \Out(G_v)$, and 
  we ask that $\rho_v(\Phi)$ belong  to $\Out^1(G_v)$.

$\bullet$ If $G_v$ is QH, it is the fundamental group of a compact surface $\Sigma$. 
Each boundary component  is associated to an incident edge or a group in $\calh\cup\calk$ (see Subsection \ref{jsj}), so $\rho_v(\Phi)$  preserves the peripheral structure of $\pi_1(\Sigma)$ and may therefore  be represented by a homeomorphism of $\Sigma$.  Since groups in $\calh\cup\calk$, and their conjugates, only intersect $G_v$ along boundary subgroups,   the image of $\Out^0(G;\calh,\calk)$ by $\rho_v$ contains the mapping class group $\calp\calm^+(\Sigma)= \Out(G_v;\mk\Inc_v ,\mk\calh_{||G_v}, \mk \calk_{||G_v})$ (see  Subsection \ref{jsj}); the index is    finite. We fix a finite index subgroup $\calp\calm^{+,1}(\Sigma)$ of type F, and we require $\rho_v(\Phi)\in \calp\calm^{+,1}(\Sigma)$. In particular,  $\Phi$ acts trivially on all boundary subgroups of $\Sigma$.

Let  $\Out^r(G;\calh ,\calk)$ consist of automorphisms $\Phi\in\Out^0(G;\calh,\calk)$ whose images $\rho_v(\Phi)$ satisfy the conditions stated above. These automorphisms act trivially on edge stabilizers. 

It follows from Lemma \ref{lem_fini} that 
$\Out^r(G;\calh ,\calk) \cap \Out(G;\mk\calh,\calk) $ always has finite index in $\Out(G;\mk\calh,\calk) $. If groups in $\calh$ are abelian, then $\Out^r(G;\calh ,\calk)$ has finite index in $\Out (G;\calh ,\calk)$.
It therefore suffices to prove that  $$O:=\Out^r(G;\calh ,\calk) \cap \Out(G;\mk\calh,\calk) $$ is of type $F$  (this argument, based on Lemma \ref{lem_fini}, is the only place where we use  the assumptions on $\calh$ and $\calk$).

Every edge of $\Tcan$ has an endpoint $v$ with $G_v $ rigid or QH, so elements of $O$ act trivially on edge stabilizers of $\Tcan$. Consider an abelian vertex stabilizer $G_v$. Elements in $\rho_v(O)$ are the identity on incident edge groups and groups in $\calh_{||G_v}$, and leave groups in $\calk_{||G_v}$ invariant. By Lemma \ref{arithm} these conditions define a group $B_v\inc\Out(G_v)$ which is of type VF, and  $C_v:=B_v\cap\Out^1(G_v)$ is a group of type F containing $\rho_v(O)$.

Recall from Subsection \ref{automs} the exact sequence 
$$1\ra \Tw \ra  \Out^0(\Tcan)
\xra{\ \rho\ }\prod_{v\in V} \Out(G_v;\Inc _v). 
   $$ 
 We claim that  the image of $O$ by $\rho$ is a direct product $\prod_{v\in V} C_v$, with 
 $C_v$  as above 
 if $G_v$ is abelian, $C_v=\calp\calm^{+,1}(\Sigma)$ if $v$ is QH, and $C_v$ trivial if $ v$ is rigid. The image is contained in the product. Conversely, given a family $(\Phi_v)_{v\in V}$, with $\Phi_v\in C_v$, the automorphisms $\Phi_v$ act trivially on incident edge groups, so there is $\Phi\in\Out^0(\Tcan)$ with $\rho_v(\Phi)=\Phi_v$.
  Since $\Phi_v$ acts trivially on $\Inc_v\cup\calh_{||G_v}$ and preserves $\calk_{||G_v}$, this automorphism is in $O$. This proves the claim.
 
 It follows that $\rho(O)$ is of type F.
 The group of twists $\calt$ is contained in $O$, because twists act trivially on vertex groups and $T$ is relative to $\calh\cup\calk$, so we can conclude that $O$ is of type F by Theorem \ref {extF} if we know that $\calt$ is of type F. 
 The group $\calt$ is a finitely generated abelian group.  
 It is torsion-free, hence of type F, as shown in Section 4 of \cite{GL6}  (alternatively, one can replace  $\Out^r(G;\calh ,\calk)$ by its intersection with a torsion-free finite index subgroup of $\Out(G)$, which exists by Corollary 4.4 of \cite{GL6}).

This proves Theorem  \ref{mcga}  in the freely indecomposable  case. To prove it in general, 
we need to study automorphisms of free products. 
 
 \subsubsection{Automorphisms of free products}
 
 In this subsection, $G$ does not have to be relatively hyperbolic.

 Let $\calg=\{G_i\}$ be a family of subgroups of $G$. We have defined $\Out(G;\calg )$  as automorphisms leaving the conjugacy class of each $G_i$ invariant, and $\Out(G;\mk\calg)$  as automorphisms   acting trivially on each $G_i$. 
 
 More generally, consider a group of automorphisms $\bQ_i\subset \Out(G_i)$
and  $\bQ=\{\bQ_i\}$. We 
would like to define $\Out(G;\mk[\bQ]\calg )\inc\Out (G;\calg )$ as the automorphisms $\Phi$ acting on each $G_i$ as an element of $\bQ_i$. To be precise, given $\Phi\in\Out (G;\calg )$, choose   representatives $\varphi_i$ of $\Phi$ in $\Aut(G)$ with $\varphi_i(G_i)=G_i$. We say that $\Phi$ belongs to $\Out(G;\mk[\bQ]\calg)$ if every $\varphi_i$ represents an element of $\bQ_i$. This is well-defined (independent of the chosen $\varphi_i$'s) if each $G_i$ is a free factor (more generally, if the normalizer of $G_i$ acts on $G_i $ by inner automorphisms). 
 
The goal of this subsection is to show:
\begin{prop} \label{fp}
 Let $G=G_1*\dots*G_n*F_p$, with $F_p $ free  of rank $p$, and let $\calg=\{G_i\}$. Assume that all groups $G_i$ and $G_i/Z(G_i)$ have type F.  
 
 Let  $\bQ=\{\bQ_i\}$ be  a family of  subgroups   $\bQ_i\inc\Out(G_i)$. If every $\bQ_i$ is  of type   VF, 
then   $\Out(G;\mk[\bQ]\calg )$ has type VF.
 
More precisely, there exists a finite index subgroup $\Out^1(G;\calg)\inc\Out(G;\calg)$, independent of $\bQ$, 
  such that, if every $\bQ_i$ is  of type  F, 
then $\Out^1(G;\calg)\cap\Out(G;\mk[\bQ]\calg)$ has type F.

\end{prop}

 The ''more precise'' assertion implies the first one, since $\Out(G;\mk[\calq']\calg)$ has finite index in $\Out(G;\mk[\calq]\calg)$ if every $\bQ'_i $ is a finite index subgroup of $ \bQ_i$.
 
 Assume that  $G_i$ and $G_i/Z(G_i)$ have type F.   The proposition  says in particular that the Fouxe-Rabinovitch group  $\Out (G;\mk\calg)$ is of type VF, and that $\Out (G;\calg ) $ is of type VF  if every $ \Out(G_i)$ is. 
If we consider the Grushko decomposition of $G$, then  $\Out (G;\calg )$ has finite index in $\Out(G)$ and we get:

  \begin{cor} [Correcting \cite{GL1}, Theorem 5.2] \label{corr}
  Let $G=G_1*\dots*G_n*F_p$, with $F_p $ free, $G_i$ non-trivial, not isomorphic to $\Z$,  and not a free product. If every $G_i$ and $G_i/Z(G_i)$ has type F, and every $\Out(G_i)$ has type VF, then $\Out(G)$ has type VF. \qed
  \end{cor}
 
\begin{proof}[Proof of Proposition \ref{fp}]

 We prove the ``more precise'' assertion, so we assume that $\calq_i\subset \Out(G_i)$ has type F.
We shall apply Theorem \ref{extF} to the action of $\Out(G;\mk[\bQ]\calg )$   on  
the outer space defined in \cite{GL1}.  We let $\cald$ be
  the Grushko deformation space 
  relative to $\calg$, \ie the JSJ deformation space of $G$ over the trivial group relative to $\calg$ (see Subsection \ref{jsj}). Trees in $\cald$ have trivial edge stabilizers, and non-trivial vertex stabilizers are conjugates of the $G_i$'s. 

Like ordinary outer space \cite{CuVo_moduli}, the projectivization $\hat \cald$ of $\cald$  is a complex consisting of simplices with missing faces, and the spine of $\hat \cald$ is a simplicial complex. It is 
  contractible for the weak topology \cite{GL2}.
  
  The group $\Out  (G;\calg )$ acts on $\cald$, hence on the spine, 
  and the action of    the Fouxe-Rabinovitch group $\Out (G;\mk\calg)\inc \Out(G;\mk[\bQ]\calg )$  is cocompact because there are finitely many possibilities for the quotient graph $T/G$, for $T\in\cald$. In order to apply Theorem \ref{extF}, we just need to show that  stabilizers are of type F.

$\Out  (G;\calg )$ also acts on the free group (isomorphic to $F_p$) obtained from $G$ by killing all the $G_i$'s (it may be viewed as the topological fundamental group of $\Gamma=T/G$, for any $T\in\cald$). 
 In other words, there is a natural map $\Out  (G;\calg )\to \Out(F_p)$. We fix a torsion-free finite index subgroup $\Out^1(F_p)\inc\Out(F_p)$, and we define $\Out ^1 (G;\calg )\inc\Out  (G;\calg )$ as the pullback of  $\Out^1(F_p)$.

Given $T\in\cald$, we let $S$ be its stabilizer for the action of $\Out ^1 (G;\calg )$, and $S_\bQ$ its stabilizer for the action of $ \Out^1(G;\calg)\cap\Out(G;\mk[\bQ]\calg )$. We complete the proof by showing that $S_\bQ$ has type F.

We first  claim that $S$ equals $\Out^0(T)$, the group of automorphisms of $G$ leaving $T$ invariant and acting trivially on $\Gamma=T/G$. Clearly $\Out^0(T)\inc S$. Conversely, we have to show that any $\Phi\in S$ acts as the identity on $\Gamma$. First, $\Phi$ fixes all   vertices   of $\Gamma$ carrying a non-trivial group $G_v$, because $G_v$ is a $G_i$ (up to conjugacy)  and the $G_i$'s are not permuted. In particular, 
  by minimality of $T$, all terminal vertices of $\Gamma$ are fixed.  Also, by our definition of $\Out ^1 (G;\calg )$, the image of $\Phi$ in $\Out(\pi_1(\Gamma))$ is trivial or has infinite order. The claim follows because any non-trivial symmetry of $\Gamma$ fixing all terminal vertices maps to a non-trivial element of finite order in $\Out(\pi_1(\Gamma))$ if $\Gamma$ is not a circle.

The map $\rho$ (see Subsection \ref{automs}) maps $S$ onto $\prod_i\Out(G_i)$, and the image of $S_\bQ$ is $\prod_i\bQ_i$, a group of type F. The kernel is the group of twists $\calt$, which is contained in $S_\bQ$, so it suffices to check that $\calt$ has type F. Since edge stabilizers are trivial, $\calt$ is a direct product $\prod_iK_i$, with $K_i=G_i^{n_i}/Z(G_i)$; here $n_i$ is the valence of the vertex carrying $G_i$ in $\Gamma$, and the center $Z(G_i)$ is embedded diagonally (see \cite{Lev_automorphisms}). There are   exact sequences  $$1\to G_i^{n_i-1}\to G_i^{n_i}/Z(G_i)\to G_i /Z(G_i)\to1,$$ so the assumptions of the proposition ensure that $\calt$ is of type F.
\end{proof}

\subsubsection{The infinitely-ended case}
 
 We can now prove Theorem  \ref{mcga} in full generality. We let $G=G_1*\dots*G_n*F_p$ be the Grushko decomposition of $G$ relative to $\calh\cup\calk$ (see Subsection \ref{jsj}), and $\calg=\{G_i\}$.
 Each $G_i$ is toral relatively hyperbolic, so has type F by \cite{Dah_classifying}. 
Its center is trivial if $G_i$ is nonabelian, so $G_i/Z(G_i)$ also has type F. This will allow us to use Proposition \ref{fp}.

\begin{lem} \label{gloloc}  
Let $\bQ=\{\bQ_i\}$, with  $\bQ_i=\Out(G_i;\mk\calh_{|G_i},\calk_{|G_i})$, and let $\bG=\{\bG_i\}$ with $\bG_i=\Out(G_i;\calh_{|G_i},\calk_{|G_i})$. 

Then  $$\Out(G;\mk[\bQ]\calg)=\Out(G;\mk \calh,\calk)\cap\Out(G;\calg),$$ and $$\Out(G;\mk[\bG]\calg)=\Out(G;\calh,\calk)\cap\Out(G;\calg).$$

Moreover, $\Out(G;\mk[\bQ]\calg)$ has finite index in $\Out(G;\mk \calh,\calk)$, and $\Out(G;\mk[\bG]\calg)$ has finite index in $\Out(G; \calh,\calk)$.
\end{lem}

\begin{proof} If $\Phi$ belongs to $\Out(G;\mk[\bQ]\calg)$, it belongs to $\Out(G;\mk \calh,\calk)$ because every group in $\calh\cup\calk$ has a conjugate contained in some $G_i$. Conversely, automorphisms in $\Out(G;\mk \calh,\calk)$ preserve the Grushko deformation space relative to $\calh\cup\calk$ and therefore permute the $G_i$'s, so $\Out(G;\calg)\cap\Out(G;\mk \calh,\calk)$ has finite index in $\Out(G;\mk \calh,\calk)$. If $\varphi\in\Aut(G)$ leaves $G_i$ invariant and maps a non-trivial $H\inc G_i$ to a conjugate $gHg\m$, then $g\in G_i$ because $G_i$ is a free factor. This shows  $$\Out(G;\mk \calh,\calk)\cap\Out(G;\calg)\inc\Out(G;\mk[\bQ]\calg),$$ completing the proof for $\Out(G;\mk[\bQ]\calg)$. The proof for $\Out(G;\mk[\bG]\calg)$ is similar.
\end{proof}

  The first assertion of Theorem \ref{mcga} now follows immediately from the one-ended case together with Proposition \ref{fp}, since $\Out(G;\mk \calh,\calk)$ contains $\Out(G;\mk[\bQ]\calg)$ with finite index.
There remains to prove  the ``moreover''.

 Each $G_i$ is freely indecomposable relative to $\calh_{|G_i}\cup\calk_{|G_i}$,  so we may apply the ``moreover'' of Theorem \ref{mcga} to $G_i$.
 We get  a finite index subgroup   $\calr_i^1\subset \calr_i $ 
such that  $\calq_i^1:=\calr_i^1 
\cap \bQ_i$ 
has type F. Let   $\calr^1=\{\calr_i^1\}$, and $\calq^1=\{\calq_i^1\}$. 

By
  Proposition \ref{fp}, there is a finite index subgroup
$\Out^1(G;\calg)\inc\Out(G;\calg)$ such that
$\Out^1(G;\calg)\cap\Out (G;\mk[\bQ^1]\calg)$ has type  $F$.
Now write
$$\Out^1(G;\calg)\cap\Out (G;\mk[\bQ^1]\calg)=\Out^1(G;\calg)\cap\Out (G;\mk[\bG^1]\calg)\cap\Out  (G;\mk[\bQ ]\calg).$$

By Lemma \ref{gloloc}, we may replace the last term $\Out (G;\mk[\bQ ]\calg)$ by $\Out(G;\mk \calh,\calk)$.
Defining $$\Out^1(G;\calh,\calk):=\Out^1(G;\calg)\cap\Out (G;\mk[\bG^1]\calg),$$  we have shown that $\Out^1(G;\calh,\calk)\cap\Out(G;\mk\calh,\calk )$ has type F. There remains to check that $\Out^1(G;\calh,\calk)$ is a  finite index subgroup of $\Out(G;\calh,\calk)$.

Since $\Out^1(G;\calg)$ has finite index in $\Out (G;\calg)$, and $\calr_i^1$ is a finite index subgroup of $\calr_i$, 
the group $\Out^1(G;\calh,\calk)$ has finite index in $\Out (G;\calg)\cap\Out (G;\mk[\bG]\calg)$, which equals $\Out (G;\mk[\bG]\calg)$ and has finite index in   $\Out(G;\calh,\calk)$ by Lemma \ref{gloloc}.

 This completes the proof of Theorem  \ref{mcga}.

\subsubsection{The action on abelian groups} \label{aag}

We study the action of $\Out(G)$ on abelian subgroups. 
The result of this subsection (Proposition \ref{image}) will be needed in Subsection \ref{finpf}.

A toral relatively hyperbolic group has finitely many conjugacy classes of non-cyclic maximal abelian subgroups. Fix a representative $A_j$ in each class. Automorphisms of $G$ preserve the set of $A_j$'s (up to conjugacy), so some finite index subgroup of $\Out(G)$ maps to $\prod_j\Out(A_j)$. 
We shall show  in particular that \emph{the image of a suitable finite index subgroup $\Out'(G)\inc\Out(G)$ is  
  a product of  McCool groups  $ \prod_j\Out(A_j; \mk{\{F_j\}} )\inc\prod_j\Out(A_j)$}. 

This product structure expresses the fact that   automorphisms of non-conjugate maximal non-cyclic abelian subgroups do not interact. Indeed, consider a family of elements $\Phi_j\in\Out(A_j)$, and suppose that each $\Phi_j$, taken individually, extends to an automorphism $\widehat\Phi_j\in\Out'(G)$; then there is $\Phi\in\Out'(G)$ inducing all $\Phi_j$'s simultaneously.

  In fact, we will  work with two (possibly empty) finite families $\calh$, $\calk$ of  abelian 
subgroups, and we will restrict to $\Out(G;\mk\calh,\calk)$. 
 We shall therefore define a finite index subgroup $\Out'(G;\mk\calh,\calk)\inc\Out(G;\mk\calh,\calk)$. 

First assume that $G$ is freely indecomposable relative to $\calh\cup\calk$. As in Subsection \ref{oe}, we consider the canonical JSJ tree $\Tcan$, we restrict to automorphisms $\Phi\in\Out(G;\mk\calh,\calk)$ acting trivially on $\Gcan=\Tcan/G$, and we  define $\Out'(G;\mk\calh,\calk)$ by imposing conditions on the action on  non-abelian vertex groups $G_v$: if $G_v$ is QH, the action should be trivial on all boundary subgroups of $\Sigma$ (i.e.\ $\rho_v(\Phi)\in\calp\calm^+(\Sigma)$); if $G_v$ is  
rigid, then  $\rho_v(\Phi)$ should be trivial. We have explained in Subsection \ref{oe}   why this defines a subgroup of finite index  $\Out'(G;\mk\calh,\calk)$ in $\Out(G;\mk\calh,\calk)$.  Note that 
$\Out'(G;\mk\calh,\calk)$ acts trivially on edge groups of $\Tcan$.

If $G$ is not freely indecomposable relative to $\calh\cup\calk$, let $G=G_1*\dots*G_n*F_p$ be the relative Grushko decomposition.   To define $\Out'(G;\mk\calh,\calk)$, we require that $\Phi$ maps $G_i$ to $G_i$ (up to conjugacy), and the induced automorphism belongs to $\Out'(G_i;\mk\calh_{ | G_i},\calk_{ | G_i})$ as   defined above. 

Elements of $\Out'(G;\mk\calh,\calk)$ leave every $A_j$ invariant (up to conjugacy), and we denote by 
$$\theta:\Out'(G;\mk\calh,\calk)\to\prod_j\Out(A_j)$$ the natural map.

We can now state:

\begin{prop} \label{image}
Let $\calh$, $\calk$  be two finite families   of  abelian 
subgroups, and let  
 $\Out'(G;\mk\calh,\calk)$ be  the finite index subgroup of $\Out(G;\mk\calh,\calk)$    defined above. 

There exist subgroups $F_j\inc A_j$ such that the image of $\theta:\Out'(G;\mk\calh,\calk)\to\prod_j\Out(A_j)$ equals  $\prod_j\Out(A_j; \mk{\{F_j\}},\calk_{ | A_j})$.
\end{prop}

Recall that the $A_j$'s are representatives of conjugacy classes of non-cyclic maximal abelian subgroups.

\begin{proof}
The $A_j$'s are contained (up to conjugacy) in factors $G_i$ of the Grushko decomposition relative to $\calh\cup\calk$, and the $G_i$'s are invariant under $\Out'(G;\mk\calh,\calk)$.
 Since any  family 
of automorphisms $\Phi_i\in\Out'(G_i;\mk\calh_{|G_i},\calk_{|G_i})$ 
extends to an automorphism 
$\Phi\in\Out'(G;\mk\calh,\calk)$, 
we may assume that $G$ is freely indecomposable relative to $\calh\cup\calk$.

Consider $\Tcan$ as above. If $A_j$ is contained in a rigid vertex stabilizer, then $\Out'(G;\mk\calh,\calk)$ acts trivially on $A_j$ and we define $F_j=A_j$. If not, $A_j$ is a vertex stabilizer $G_v$. Vertex stabilizers adjacent to $v$ are rigid or QH, and because of the way we defined $\Out'(G;\mk\calh,\calk)$ it leaves $A_j$ invariant and acts   trivially  on incident edge groups. It also acts trivially on the groups belonging to $\calh _{ | A_j}$.

Defining $F_j$ as the subgroup of $A_j$ generated by incident edge groups and groups in $\calh _{ | A_j}$, we have proved that the image of $\theta$ is contained in $\prod_j\Out(A_j; \mk{\{F_j\}},\calk_{ | A_j})$.
Conversely, choose a family $\Phi_j\in\Out(A_j; \mk{\{F_j\}},\calk_{ | A_j})$. As explained in Subsection \ref{automs}, there exists $\Phi\in\Out^0(\Tcan)$ acting trivially on cyclic, rigid, and QH vertex stabilizers, and inducing $\Phi_j $ on $A_j$. 
We  check    that $\Phi$ acts trivially on any $H\in\calh$. 
Such a group $H$ fixes a vertex $v\in \Tcan$. If $G_v$ is cyclic, rigid, or QH, the action of $\Phi$ on $H$ is trivial. If not, $G_v$ is (conjugate to) an $A_j$ and the action is trivial because $H\inc F_j$. 
A similar argument shows that $\Phi$ preserves $\calk$ up to conjugacy, so $\Phi\in \Out(G;\mk\calh,\calk)$.
Since $\Phi$ acts trivially on rigid  
and QH vertex stabilizers, $\Phi\in \Out'(G;\mk\calh,\calk)$.
\end{proof}

\subsection{Automorphisms preserving a tree}

We now study the stabilizer of a tree. The following theorem clearly implies Theorem \ref{mct}. 

 \begin{thm} \label{mcgg}
Let $G$ be a toral relatively hyperbolic group.  Let $T$ be a simplicial tree on which $G$ acts with abelian edge stabilizers. Let 
$\calk$ be a finite family of abelian subgroups of $G$,  each of which fixes a point in $T$.
Then $\Out(T,\calk)=\Out(T)\cap\Out(G;\calk)$ 
is of type VF.
\end{thm}

The group $\Out(T,\calk)$ is the subgroup of $\Out(G)$ consisting of automorphisms leaving $T$ invariant and mapping each group of $\calk$ to a conjugate (in an arbitrary way). 
    The tree $T$ is assumed to be minimal, but it may be a point, it may have trivial edge stabilizers, and non-cyclic abelian subgroups need not be elliptic.

Theorem \ref{mcga}  proves Theorem \ref{mcgg} when $T$ is  a point. 
Also note that, if $G$ is abelian and $T$ is not a point, then $T$ is a line on which
$G$ acts by  integral  translations, 
and  $\Out(T,\calk)$ is of type VF because it equals $\Out(G;\calk\cup \{N\})$, with
  $N$   the kernel of the action of $G$ on $T$.

Thus, we assume from now on that $G$ is not abelian. 
We will prove Theorem \ref{mcgg} when $T$ has cyclic edge stabilizers before treating the general case.  This special case is much easier because $\Out(G_e)$ is finite for every edge stabilizer $G_e$, and we may apply Proposition 2.3 of   \cite{Lev_automorphisms}.

\subsubsection{Cyclic edge   stabilizers} \label{cyc}

In this subsection we prove  Theorem \ref{mcgg} when all edge stabilizers $G_e$ of $T$ are cyclic (possibly trivial); this happens in particular if $G$ is hyperbolic.   

As in Subsection \ref{automs}, we consider the exact sequence 
$$1\to\calt\to\Out^0(T)\xra{\ \rho\ }\prod_{v\in V}\Out(G_v;\Inc_v) .$$
The image of $\rho$ contains $\prod_{v\in V}\Out(G_v;\mk\Inc_v)$, and the index is finite because 
all groups $\Out(G_e)$ are finite (see \cite{Lev_automorphisms}, where $\Out(G_v;\mk\Inc_v)$  is denoted $PMCG(G_v)$).
The preimage of $\prod_{v\in V}\Out(G_v;\mk\Inc_v)$ is thus a finite index subgroup $\Out^1(T)\inc\Out(T)$.

We want to prove that
$\Out (T,\calk)$ is of type VF, so we restrict the preceding discussion to $\Out (T,\calk)$.
Let 
$$\Out^1(T,\calk)=\Out^1(T)\cap\Out(G;\calk),$$ a finite index subgroup. We show that $\Out^1(T,\calk)$ is of type VF  (this will not use   the assumption that edge stabilizers are cyclic).

The image of $\Out^1(T,\calk) $ by $\rho$ is contained in $\prod_{v\in V}\Out(G_v; \mk\Inc_v, \calk_{||G_v})$, with $\calk_{||G_v}$ as in Subsection \ref{tre}, and arguing as in Subsection \ref{automs} one sees that equality holds.
On the other hand, $\Out^1(T,\calk)$
  contains $\calt$ because twists act trivially on vertex stabilizers, hence on $\calk$ since  groups of $\calk$ are elliptic in $T$. 
We therefore  have an exact sequence
$$1\to\calt\to\Out^1(T,\calk)\to\prod_{v\in V}\Out(G_v; \mk\Inc_v, \calk_{||G_v})\to1 .$$

Vertex stabilizers are toral relatively hyperbolic, so the product  is of type VF by Theorem   \ref{mcga} applied to the $G_v$'s.
We conclude the proof by showing that  $\calt$ is of type F. This will imply that $\Out^1 (T,\calk)$, hence $\Out (T,\calk)$, is VF.

We claim that $\calt$ is isomorphic to the direct product of a finitely generated abelian group   and a finite number of copies of non-abelian
vertex groups $G_v$.
We use the presentation of $\calt$ given in Proposition 3.1 of \cite{Lev_automorphisms}. 
It says that 
$\calt$ can be written as a quotient  $$\calt=\prod_{e,v} Z_{G_v}(G_e)/\grp{\calr_V,\calr_E},$$
the product being taken over all pairs $(e,v)$ where $e$ is an edge incident to $v$;
here $\calr_E=\prod_e Z(G_e)$ is the group of edge relations,
and  $\calr_V=\prod_{v} Z(G_v)$ is the group of vertex relations, both embedded naturally in $\prod_{e,v} Z_{G_v}(G_e)$.
Every group $Z_{G_v}(G_e)$ is abelian, unless $G_e$ is trivial and $G_v$ is non-abelian.
In this case $Z_{G_v}(G_e)=G_v$, and it is not affected by the edge and vertex relations since 
both $Z(G_v)$ and $Z(G_e)$ are trivial.
Our claim follows.

 It follows that $\calt$ 
is of type F provided that it is torsion-free. One may show that this is always the case, but it is simpler to replace $\Out^1(T,\calk)$ by its intersection with a torsion-free finite index subgroup of $\Out(G)$.

\subsubsection{Changing $T$}
We shall now prove Theorem \ref{mcgg} in the general case.

The first step, carried out in this subsection,  is to replace $T$ by a better tree $\hat T$ (satifsfying the second assertion of the lemma below). When all edge stabilizers are non-trivial, $\hat T$ may be viewed as the smallest common refinement  (called  lcm in   \cite{GL3b}) of $T$ and its tree of cylinders (see Subsection \ref{jsj}). 
Here is the 
construction of $\hat T$.

Consider edges of $T$ with non-trivial stabilizer. We say that two such edges belong to the same cylinder if their stabilizers commute. Cylinders are subtrees and meet in at most one point. A vertex $v$ with all incident edge groups trivial belongs to no cylinder. Otherwise $v$ belongs to one cylinder if $G_v$ is abelian, to infinitely many cylinders if $G_v$ is not abelian. 
To define $\hat T$, we shall refine $T$ at vertices $x$ belonging to infinitely many cylinders.

Given such an $x$, let $S_x$ be the set of cylinders $Y$ such that $x\in Y$. We replace $x$ by the cone $T_x$ on $S_x$: there is a central vertex, again denoted   by $x$, and vertices $(x, s_Y)$ for  $Y\in S_x$, with  an edge between $x$ and $(x,s_Y)$. Edges $e$ of $T$ incident  to $x$ are attached to $T_x$ as follows: if the stabilizer of $e$ is trivial, we attach it to the central vertex $x$; if not, $e$ is contained in a cylinder $Y$ and we attach $e$ to the   vertex $(x,s_Y)$, noting that $G_e$ leaves $Y$ invariant.

Performing this operation at each $x$ belonging to infinitely many cylinders yields a tree $\hat T$. The construction being canonical, there is   a natural  action of $G$ on $\hat T$, and $\Out(T)\inc\Out(\hat  T)$.

\begin{lem}\label{norm}
\begin{enumerate}
\item Edge stabilizers of $\hat T$ are abelian, $\hat T$ is dominated by $T$, and $\Out(\hat T)=\Out(T)$.
\item   Let $G_v$ be a non-abelian vertex stabilizer of $\hat T$. Non-trivial incident edge stabilizers $G_e$ are maximal abelian subgroups of $G_v$. If $e_1$, $e_2$ are edges   of $\hat T$ incident to $v$ with $G_{e_1}$, $G_{e_2}$ equal and  non-trivial,   then $e_1=e_2$. 
\end{enumerate}
\end{lem}

\begin{proof}
Let $Y$ be a cylinder in $S_x$  (viewed as a subtree of $T$). 
 The setwise stabilizer $G_Y$ of $Y$ is the maximal abelian subgroup of $G$ containing    stabilizers of   edges of $Y$. The stabilizer of the vertex $(x,s_Y)$ of $\hat T$, and also of the edge between $(x,s_Y)$ and $x$,   is $G_x\cap G_{Y}$; it is non-trivial (it contains the stabilizer of edges of $Y$  incident to $x$), and is a maximal abelian subgroup of $G_x$. This proves that  edge stabilizers of $\hat T$  are abelian, since the other edges have the same stabilizer as in $T$. 
 
 Every vertex stabilizer of $T$ is also a vertex stabilizer  of $\hat T$, so $T$ dominates $\hat T$. Edges of $\hat T$ which are not edges of $ T$  (those between $(x,s_Y)$ and $x$) are characterized as those having non-trivial stabilizer and having an endpoint $v$ with $G_v$ non-abelian. One recovers $T$ from $\hat T$ by collapsing these edges, so $\Out(\hat T)\inc\Out(T)$.
 
  Consider two edges $e_1$ and $e_2$ incident to $v$ in  $\Hat T$, with the same non-trivial stabilizer. They  join $v$ to vertices $(v,s_{Y_i})$, and we have seen that $G_{e_1}=G_{e_2}$ is maximal abelian in $G_v$. The groups $G_{Y_1}$ and $G_{Y_2}$ are equal because they both contain $G_{e_1}=G_{e_2}$. Edges of $Y_i$ have stabilizers contained in $G_{Y_i}$, so have commuting stabilizers. Thus $Y_1=Y_2$, so $e_1=e_2$.
\end{proof}

\begin{rem}\label{nonconj}
If $G_{e_1}$, $G_{e_2}$  are conjugate in $G_v$, rather than equal,   we conclude that  $e_1$ and $e_2$ belong to the same $G_v$-orbit.  On the other hand, edges belonging to different $G_v$-orbits may have stabilizers which are conjugate in $G$ (but not in $G_v$). 
\end{rem}

\subsubsection {The action on edge groups}
\label{sec_edge}

 In Subsection \ref{cyc} we could neglect the action of $\Out^0(T)$ on edge groups because all groups $\Out(G_e)$ were finite. We now allow edge stabilizers of arbitrary rank, so we must take these actions into account.
  We denote $\Out^0(T,\calk)=\Out^0(T)\cap\Out(G;\calk)$.

  Recall that, for each edge $e$ of $\Gamma=T/G$,
there is a natural map $\rho_e:\Out^0(T)\ra \Out(G_e)$ (see Section \ref{automs}).
The collection of all these maps defines a map 
$$\psi:\Out^0(T,\calk)\to\prod_{e\in E}\Out(G_e),$$ 
  the product being over all non-oriented edges of $\Gamma$. 
 We denote by $Q$ the image of $\Out^0(T,\calk)$ by 
 $\psi$, so that we have the exact sequence
$$1\ra \ker\psi\ra \Out^0(T,\calk)\ra Q\ra 1.$$
 
\begin{lem} \label{pt2}
 If $T$ satisfies the second assertion of Lemma \ref{norm}, 
  then the group $Q$ is  of type VF.
\end{lem}

This lemma will be proved in the next subsection. 
We first explain how to deduce  Theorem \ref{mcgg} from it. The first assertion of Lemma \ref{norm} implies that the theorem  holds for $T$ if it holds for $\hat T$, so we may assume that $T$ satisfies the second assertion of  Lemma \ref{norm}.

The kernel of $\psi$ is the group discussed in Subsection \ref {cyc} under the name   $\Out^1(T, \calk)$, but now (contrary to Convention \ref{1}) $\Out^1(T, \calk)$ may   be of infinite index in $\Out(T,\calk)$; indeed, 
 $\Out(T,\calk)$ is virtually an extension of 
$\Out^1(T ,\calk)$ by  $Q$.  To avoid confusion, we use the notation  $\ker \psi$ rather than $\Out^1(T ,\calk)$. 

We proved  in Subsection \ref {cyc} 
that $\ker \psi$ is of type VF, and $Q$ is of type VF by the lemma,   but this is not quite sufficient (see Remark \ref{subt}).
  We  shall  now 
 construct a finite index subgroup   $\Out^2(T,\calk)\subset \Out^0(T,\calk)$
such that $\ker\psi\cap\Out^2(T,\calk)$ has type $F$.   Applying Corollary \ref{cor_extensionVF} to $\Out^0(T,\calk)$ then completes the  proof of Theorem \ref{mcgg}.

We argue as in Subsection \ref{cyc}. Recall from Subsection \ref{automs} 
 the   exact sequence
$$1\to\calt\to\Out^0(T,\calk)\xra{\ \rho\ }\prod_{v\in V}\Out(G_v; \Inc_v, \calk_{||G_v})  $$
whose restriction to $\ker \psi$ is the exact sequence
$$1\to\calt\to\ker \psi\xra{\ \rho\ }\prod_{v\in V}\Out(G_v; \mk\Inc_v, \calk_{||G_v})\ra 1. $$

Using the ``more precise'' statement of Theorem \ref{mcga} 
we get, for each $v\in V$, 
a finite index subgroup
$\Out^1(G_v; \Inc_v, \calk_{||G_v})\inc\Out(G_v; \Inc_v, \calk_{||G_v})$ 
such that 
$$\Out^1(G_v; \Inc_v, \calk_{||G_v})\cap\Out(G_v; \mk\Inc_v, \calk_{||G_v})$$ is of type F. 
Define the finite index subgroup $\Out^2(T,\calk)\subset \Out^0(T,\calk)$
as the preimage of $\prod_{v\in V}\Out^1(G_v; \Inc_v, \calk_{||G_v})$ under $\rho$,
intersected with a torsion-free finite index subgroup of $\Out(G)$. 

Restricting the exact sequence above, 
we get an exact sequence 
$$1\to\calt'\to\ker\psi\cap\Out^2(T,\calk)\xra{\ \rho\ }  L \ra 1$$
where $L$ has finite index in the product of the groups 
$$\Out^1(G_v; \Inc_v, \calk_{||G_v})\cap \Out(G_v; \mk\Inc_v, \calk_{||G_v}) ,$$ hence has type F.
The group  $\calt'$ is a torsion-free subgroup of finite index of $\calt$, so
has type F as in Subsection \ref{cyc}.   We conclude that $\ker\psi\cap\Out^2(T,\calk)$ has type F.
As explained above, this completes the  proof of Theorem \ref{mcgg} (assuming Lemma \ref{pt2}).

\subsubsection {Proof of Lemma \ref{pt2}} \label{finpf}

There remains to prove Lemma \ref{pt2}. We let   $E_j$ be representatives of conjugacy classes of  maximal abelian subgroups containing a non-trivial edge stabilizer. 
Note that    $E_j$ is allowed to  be cyclic, and   maximal abelian subgroups of $G$ 
containing no non-trivial $G_e$ are not included.

Inside each $E_j$ we let $B_j$ be the smallest direct factor containing all edge groups included in $E_j$ (it equals $E_j$ if $E_j$ is cyclic).  It is elliptic in $T$, because it is  an abelian group generated (virtually) by elliptic subgroups. 

  Each automorphism $\Phi\in \Out^0(T,\calk)$ induces an automorphism of $E_j$, which preserves $B_j$ and all the
 edge groups it contains.
This defines a map  
$$\psi':\Out^0(T,\calk)\ra \prod_j \Out(B_j)$$
having the same kernel as the map $\psi:\Out^0(T,\calk)\to\prod_{e\in E}\Out(G_e)$ defined in Subsection \ref{sec_edge}.
Thus, it suffices to  prove that the image of $\Out^0(T,\calk)$ by $\psi'$ is of type VF. We do so by finding a finite index subgroup
$ \Out^1(T,\calk)$  (not the same as in Subsection \ref{cyc}) whose image is a product $\prod_j Q_j$ with each $Q_j$ of type VF.

Consider a non-abelian vertex  group $G_v$. Define $\Inc_{v,\Z}\inc\Inc_v$ by keeping only the incident edge groups which are infinite cyclic,   and denote by $E_{nc}(v)$   the set of edges $e$ of $\Gamma$ with origin $v$ and $G_e$ non-cyclic (if $e$ is a loop, we subdivide it so that it counts twice in $E_{nc}(v)$). By Lemma \ref{norm}  and Remark \ref{nonconj}, the edge groups $G_e$, for $e\in E_{nc}(v)$, are  non-conjugate maximal abelian  subgroups of $G_v$.

We apply Proposition \ref{image},  describing the action on  non-cyclic  maximal abelian subgroups, to $\Out(G_v;\mk\Inc_{v,\Z},\calk_{||G_v}) $.
We   get a subgroup of finite index $\Out'(G_v;\mk\Inc_{v,\Z},\calk_{||G_v})$,
and a subgroup  $F_e^v \subset G_e$ for each edge $e\in  E_{nc}(v)$,
 such that 
 the image of $\Out'(G_v;\mk\Inc_{v,\Z},\calk_{||G_v})$ in $\prod_{e\in E_{nc}(v)}\Out(G_e)$ is
 the product $\prod_{e\in E_{nc}(v)} \Out(G_e;\mk {\{F_e^v\}},\calk_{|G_e})$.

We let $\Out^1(T,\calk)\inc\Out^0(T,\calk)$ be the  subgroup consisting of automorphisms acting trivially on cyclic edge stabilizers,  and acting on non-abelian vertex stabilizers as an element of $\Out'(G_v;\mk\Inc_{v,\Z},\calk_{||G_v}) $. It has  finite index because 
$$\Out(G_v;\mk\Inc_{v,\Z},\calk_{||G_v}) \inc \rho_v(\Out^0(T,\calk))\inc\Out(G_v;\Inc_{v,\Z},\calk_{||G_v} ),$$ with all indices finite.

We now define $Q_j\inc\Out(B_j)$ as consisting of automorphisms $\Phi_j$ such that:

 \begin{enumerate}
\item if $G_e$ is a cyclic edge stabilizer contained in $B_j$, then $\Phi_j$ acts trivially on $G_e$;

\item if $B_j$ contains a non-cyclic $G_e$, and $v$ is an endpoint of $e$  with $G_v$ non-abelian, then $\Phi_j$ acts trivially on $F_e^v$;

\item non-cyclic edge stabilizers, and abelian vertex stabilizers, contained in $B_j$ are $\Phi_j$-invariant;

\item $\Phi_j$ extends to an automorphism  of $E_j$ leaving $\calk_{ | E_j}$ invariant; in particular, subgroups of $B_j$ conjugate to a group of $\calk$ are $\Phi_j$-invariant.

\end{enumerate}

This definition was designed so that the image of  $\Out^1(T,\calk)$ by $\psi'$ is contained in $\prod_j  Q_j $. We claim that equality holds: 

\begin{lem}
The image of  $\Out^1(T,\calk)$ by $\psi'$ equals $\prod_j  Q_j $. 
\end{lem}

\begin{proof}
We     fix automorphisms $\Phi_j\in Q_j\inc\Out(B_j)$, and we have to construct an automorphism  $\Phi\in\Out^1(T,\calk)$. By items 1 and 3 above, the $\Phi_j$'s induce automorphisms  $\Phi_e$ of edge stabilizers
(each non-trivial edge group $G_e$ lies in a unique  $E_j$, so there is no ambiguity in the definition of $\Phi_e$). 
   As explained after Lemma \ref{modif}, it suffices to find automorphisms $\Phi_v$ of vertex groups inducing the $\Phi_e$'s.   
We distinguish several cases.

If $G_v$ is contained in some $B_j$ (up to conjugacy), it is $\Phi_j$-invariant by item 3, so we let $\Phi_v$ be the restriction. 

If $G_v$ is abelian, but not contained in any $B_j$, we may assume that some incident $G_e$ is   non-cyclic (otherwise we let $\Phi_v$ be the identity). 
This $G_e$ is contained in some $B_j$, and $G_v\inc E_j$. In fact $G_v=E_j$: since   $G_v$ is not contained in $ B_j$, 
 it  fixes only $v$, and $E_j$ fixes $v$ because it commutes with $G_v$. We may thus  extend $\Phi_j$ to $G_v$ using item 4.

If $G_v$ is not abelian, 
we construct $\Phi_v$ in $\Out'(G_v;\mk\Inc_{v,\Z},\calk_{||G_v} ) $ as follows.
If  $e\in E_{nc}(v)$, 
the automorphism $\Phi_e$ acts trivially on $F_e^v$ by item 2, and preserves $\calk_{|G_e}$ by item 4. 
Thus, the collection of automorphisms $\Phi_e$ 
lies in 
$\prod_{e\in E_{nc}(v)} \Out(G_e;\mk {\{F_e^v\}},\calk_{|G_e})$.
    Proposition \ref{image} guarantees that   
  $\Out'(G_v;\mk\Inc_{v,\Z},\calk_{||G_v} ) $ 
  contains an automorphism $\Phi_v  $
inducing $\Phi_e$ for all $e\in E_{nc}(v)$ (and acting trivially on all cyclic incident edge groups).

We have now constructed automorphisms $\Phi_v\in\Out(G_v)$ inducing the $\Phi_e$'s, so Lemma \ref{modif} provides an automorphism $\Phi\in\Out^0(T)$ whose image in $\prod_j\Out(B_j)$ is the product of the $\Phi_j$'s because $B_j$ is virtually generated by edge stabilizers. 
We show $\Phi\in\Out^1(T,\calk)$. By construction it acts trivially on cyclic edge groups and acts on non-abelian vertex stabilizers as an element of $\Out'(G_v;\mk\Inc_{v,\Z},\calk_{||G_v}) $. We just have to check that $\Phi$ leaves any $K\in\calk $ invariant.

The group  $K$ is contained in some $G_v$. 
If $K$ is contained in some $B_j$,   it  is $\Phi$-invariant by item 4.
Otherwise, $K$ fixes no edge. 
If $G_v$ is abelian, we have seen that either all incident edge groups are  cyclic  (and $\Phi_v$ is the identity), or $G_v$ equals some $E_j$ and our choice of $\Phi_v$ using item 4 guarantees that $K$ is invariant.  
If $G_v$ is not abelian, then $K$ belongs to  $\calk_{||G_v}$ because it  fixes no edge.
 It is invariant because we chose $\Phi_v\in  \Out'(G_v;\mk\Inc_{v,\Z},\calk_{||G_v}) $.
\end{proof}

We have seen that the group $Q$ of Lemma \ref{pt2} is isomorphic to the image of $\Out^0(T,\calk)$ by $\psi'$, hence contains $\prod_j Q_j$ with finite index. 
To show that $Q$ is of type VF, there remains to show that  each  $Q_j$ is of type VF.

We defined $Q_j$ inside $\Out(B_j)$ by four conditions. As in Lemma \ref{arithm}, the first three define an arithmetic group. 
To deal with the fourth one, we consider the group $\tilde Q_j$ consisting of automorphisms of $E_j$ leaving $B_j$ and $\calk_{ | E_j}$ invariant, with the restriction to $B_j$    satisfying the first three conditions.  This is an arithmetic group. 
 It consists of block-triangular matrices, and one obtains $Q_j$ by considering the upper left blocks of matrices in $\tilde Q_j$. It follows that $K_j$ is arithmetic, as the image of an arithmetic group by a rational homomorphism   \cite[Th.\ 6]{Borel_density},
hence of type VF by Lemma \ref{arithm}.
 
  This completes the proof of Lemma \ref{pt2}, hence of Theorem \ref{mcgg}.

  \section{A finiteness result for trees}
  
  The goal of this section is  Proposition \ref{edstab},  which gives a uniform bound    for the size of certain sets   of relative JSJ decompositions of $G$. 
This an essential ingredient  in the proof of the chain condition for McCool groups.  
We will have to restrict  to root-closed (RC) trees, which are introduced in Definitions \ref{rct} and \ref{rcj}
  (they are closely related to the  primary splittings of  \cite{DaGr_isomorphism}).

  \begin{dfn} \label{rootcl} Let $H $ be a  subgroup of  a  group 
 $  G$. Its \emph{root closure} $e(H,G)$, or simply $e(H)$, is  the set of elements of $G$ having a power in $H$.  If $e(H)=H$, we say that $H$ is \emph{root-closed}.
   \end{dfn}
   
  If $G$ is toral relatively hyperbolic and $H$ is abelian, $e(H)$   is a direct factor of the maximal abelian subgroup containing $H$, and $H$ has finite index in $e(H)$.
  Also note that, given $h\in G$ and $n\ge2$, there exists at most one element $g$ such that $g^n=h$.
  
  The following fact is completely general.
  
  \begin{lem} \label{ev}
  Let $T$ be a tree with an action of an arbitrary group. The following are equivalent:
  \begin{itemize}
  \item Vertex stabilizers of $T$ are root-closed.
  \item Edge stabilizers of $T$ are root-closed.
  \end{itemize}
 \end{lem}
 
\begin{proof}  If $g^n$ fixes an edge $e=vw$, it fixes $v$ and $w$. If vertex stabilizers are root-closed, $g$ fixes $v$ and $w$, hence fixes $e$, so edge stabilizers  are root-closed.

Conversely, if $g^n$ fixes a vertex $v$, then $g$ is elliptic hence fixes a vertex $w$. Edges between $v$ and $w$ (if any) are fixed by $g^n$, hence by $g$ if edge stabilizers are root-closed. Thus  $g$ fixes $v$.
 \end{proof}
 
We now go back to  a toral relatively hyperbolic group $G$.  
  \begin{dfn} \label{rct}
  A tree $T$ is an   \emph{$\RC$-tree} if:
 \begin{itemize}
  \item    all non-cyclic abelian subgroups fix a point in $T$;

  \item edge stabilizers of $T$ are abelian and root-closed.

    \end{itemize}
  \end{dfn}
  
  When $G$ is hyperbolic, $\RC$-trees are the $\Zmax$-trees of \cite{DG2}: non-trivial edge stabilizers are maximal cyclic subgroups.

  \begin{lem} \label{ecl} 
  \begin{enumerate}
  \item Let $T$ be an $\RC$-tree with all edge stabilizers non-trivial. Its 
tree of cylinders $T_c$ (see Subsection \ref{jsj}) 
is an $\RC$-tree belonging to the same deformation space as $T$.
  \item
  If $T_1$ and $T_2$ are $\RC$-trees relative to some family $\calh$, and edge stabilizers of $T_1$ are elliptic in $T_2$, there is 
  an $\RC$-tree $\hat T_1$ relative to   $\calh$ which refines $T_1$ and dominates $T_2$. Moreover, the stabilizer of any edge of $\hat T_1$ fixes an edge in $T_1$ or in $T_2$.
    \end{enumerate}
  \end{lem}
  
  \begin{proof} 
Non-triviality of  edge stabilizers ensures that $T_c$ is defined.
  The vertex stabilizers of $T_c$ are vertex stabilizers of $T$ or maximal abelian subgroups, so are root-closed. The deformation space does not change because $T$ is relative to non-cyclic abelian subgroups (see \cite{GL4}, prop.\  6.3). This proves 1.
  
   We define  a refinement $\hat T_1$ of $T_1$ dominating $T_2$ as in Lemma 3.2 of \cite{GL3a},  by blowing up each vertex $v$ of $T_1$ into  a $G_v$-invariant subtree of $T_2$. We just have to check that its edge stabilizers are root-closed. 
  As in the proof of Lemma 4.9 of \cite{DG2}, an edge stabilizer of $\hat T_1$ is an edge stabilizer of $T_1$ or is the intersection of a vertex stabilizer of $T_1$ with an edge stabilizer of $T_2$, so is  root-closed.
  \end{proof}

   \begin{prop} \label{access}
  Let $G$ be toral relatively hyperbolic. 
  In each of the following two cases, there is a   bound for the number of orbits of edges of a minimal   tree $T$ with abelian edge stabilizers:
  \begin{enumerate}
  \item $T$ is bipartite: each edge has exactly one endpoint with abelian stabilizer (redundant vertices are allowed); 
  \item $T$ is an $\RC$-tree with no redundant vertex.
  \end{enumerate}
  \end{prop}

Here, and below, the bound has to depend only      on $G$ (it is independent of the trees under consideration).  
  
  Case 1 applies in particular to trees of cylinders.

  \begin{proof} We cannot apply Bestvina-Feighn's accessibility theorem \cite{BF_bounding} directly because $T$ does not have to be reduced in the sense of \cite{BF_bounding}: $\Gamma=T/G$   may have a vertex $v$ of   valence 2 such that an incident 
  edge carries the same group as $v$.  We say that such a $v$ is a non-reduced vertex. The assumptions   
   rule out the possibility that $\Gamma$ contains long segments consisting of   non-reduced vertices   (as in the example on top of page 450 in \cite{BF_bounding}).

If $T$ is bipartite, consider all non-reduced vertices of $\Gamma$ and collapse  exactly one of the incident edges. This yields a reduced graph of groups, and at most half of the edges of $\Gamma$ are collapsed, so \cite{BF_bounding} gives a bound.

If $T$ is an $\RC$-tree with no redundant vertex, every non-reduced vertex $v$ of $\Gamma=T/G$ has exactly two adjacent edges $e_v$ and $f_v$, whose groups satisfy $G_{e_v}\incs G_v=G_{f_v}$. 
  Among all edges incident to a non-reduced vertex, consider the set  $E_m$ consisting of those with $G_e$ of minimal rank.
No two edges of $E_m$ are adjacent at a non-reduced vertex, because $T$ is an $\RC$-tree. Now collapse the edges in $E_m$. 

  If $I=e_1\cup e_2\cup \dots \cup e_k$ is a  maximal segment in the complement of the set of vertices of $\Gamma$ having degree 3 or carrying a non-abelian  group, we never collapse adjacent edges $e_i,e_{i+1}$ (and we do not collapse $e_1$ if $k=1$; we may collapse $e_1$ and $e_3$ if $k=3$). 
 It follows that 
at least one third of the edges of $\Gamma$ remain after the   collapse.

Repeat the process. 
Denote by $M$ the maximal rank of abelian subgroups of $G$. 
After at most $M$    
  steps one obtains a graph of groups which is reduced in the sense of \cite{BF_bounding}, hence has at most $N$ edges for some fixed $N$. The number of edges of $\Gamma$ is bounded by $3^MN$. 
  \end{proof}

  \begin{prop} \label{dom2}
  Given a   toral relatively hyperbolic group $G$, 
   there exists a number $M$ such  that, if $T_1\to T_2\to\dots\to T_p$ is a sequence of maps between $\RC$-trees  
   belonging to distinct deformation spaces,  then $p\le M$.
  \end{prop}

  \begin{proof}  
  
  There are two steps.
  
  $\bullet$ The first step is to reduce to the case when no edge stabilizer   is trivial. 
  Consider the tree $\bar T_i$ (possibly a point) obtained from $T_i$ by collapsing all edges with non-trivial stabilizer. A   map $  T_i\to  T_{i+1}$ cannot send an arc with non-trivial stabilizer to the interior of an edge with trivial stabilizer, so $\bar T_i$ dominates $\bar T_{i+1}$. Vertex stabilizers of $\bar T_i$  are free factors, there are finitely many possibilities for their isomorphism type.

  Using Scott's complexity, it is shown in Section 2.2 of \cite{Gui_actions} that the number of times that the deformation space $\cald_i$ of $\bar T_i$ differs from that of  $\bar T_{i+1}$  is   uniformly bounded. We may therefore assume that $\cald=\cald_i$  is independent of $i$. 
  
   Let 
 $H_1,\dots,H_k$ be representatives of   conjugacy classes of non-trivial vertex stabilizers of trees in $\cald$.    
  They 
  are free factors of $G$, hence toral relatively hyperbolic, 
  and $k$ is bounded. 
  
  Consider the action of $H_j$ on its minimal subtree $T^j_i\inc T_i$  (we let $T^j_i$ be any fixed point if the action is trivial). It is an  $\RC$-tree, and   no edge stabilizer is trivial. The deformation space of $T_i$ is completely determined by $\cald
  $ and the deformation spaces $\cald ^j_i$ of the trees $T^j_i$ (viewed as trees with an action of $H_j$). It therefore suffices to bound (by a constant depending only on $H_j$) the 
 number of times that $\cald ^j_i$ changes in a 
  sequence $T_1^j\to T_2^j\to\dots\to T_p^j$, so we may continue the proof under the additional assumption that the $T_i$'s have non-trivial edge stabilizers.

  $\bullet$ Now that edge stabilizers are non-trivial, the tree of cylinders of $T_i$ is defined. By the first assertion of Lemma \ref{ecl},
  we may assume that it equals $T_i$.

  Since all trees are trees of cylinders,  Proposition  4.11 of \cite{GL4}
lets us assume that all    
  domination maps  $  T_i\to  T_{i+1}$   send   vertex to vertex, and   map  an edge to either a point or an edge. Such a map  may collapse an edge to a point, or identify edges belonging to different orbits, or identify edges in the same orbit. The first two phenomena are easy to control since they decrease the number of orbits of edges;     controlling the third one requires more care (and restricting to $\RC$-trees). 

We associate a complexity $(n,-s)$ to each $T_i$, with $n$  the number 
 of edges of $T_i/G$, and $s$ the sum of the ranks of its   edge groups;
complexities are ordered lexicographically. We claim that the complexity of $T_{i+1}$ is strictly smaller than that of $T_i$. 
This  gives the required uniform bound on $p$, since $n$ (hence also   $s$) is   bounded 
by the first case of Proposition \ref{access}.

Let $f_i:T_i\to T_{i+1}$ be a domination map as above.   Complexity clearly cannot increase when passing from $T_i$ to $T_{i+1}$.  If $n$ does not decrease, no edge of $T_i$ is collapsed in $T_{i+1}$. Since $ T_i$ and $T_{i+1}$ belong to distinct deformation spaces, there exist distinct  edges $e,e'$ identified by $f_i$. 
They  have to belong  to the same orbit (otherwise $n$ decreases),
so $e'=ge$ for some $g\in G$. 
 The group $\grp{g,G_e}$ fixes the edge $f_i(e)=f_i(e')$ of $T_{i+1}$, so is abelian. 
It has rank bigger than the rank of $G_e$ because $G_e$ is root-closed and $g\notin G_e$.
Thus  $s$ increases, and   the complexity decreases. 
  \end{proof}

Let $\cala $ be the family of all abelian subgroups. Let $\calh$ be a family of subgroups of $G$.   
A  JSJ tree (over $\cala$) relative to $\calh $ may be defined as a tree $T$ such that
$T$  is relative to $\calh$, edge stabilizers of $T$ are
elliptic in every   tree   which is relative  to $\calh$, and $T$ dominates every tree satisfying the previous conditions  (all trees 
are assumed to have abelian edge stabilizers).
This motivates the following definition, where we require that $T$ be an $\RC$-tree (compare Section 4.4 of \cite{DG2}). Recall that $\hp$ is obtained by adding  
all non-cyclic abelian subgroups to $\calh$.

\begin{dfn} \label {rcj}
  Let $G$ be a toral relatively hyperbolic group, 
and $\calh$   a family of subgroups. 
A tree $T$ is an \emph{$\RC$-JSJ tree relative to $\hp$} if:
\begin{enumerate}
\item $T$ is relative to $\hp$, and is an $\RC$-tree; 
\item edge stabilizers of $T$ are
elliptic in every   tree   with abelian edge stabilizers (not necessarily an $\RC$-tree) which is relative to $\hp$; 
\item $T$ dominates every tree satisfying 1 and 2.
\end{enumerate}
\end{dfn}

We will  construct $\RC$-JSJ trees in Section \ref{pfcc}.  Note that non-cyclic edge stabilizers always satisfy 2.

\begin{prop} \label{edstab}
Let $G$ be a   toral relatively hyperbolic group. 
Let  $\calh_1\inc\dots\inc \calh_i\inc\dots$ be an increasing sequence (finite or infinite) of families of subgroups, with $G$ freely indecomposable relative to $\calh_1$. 
For each $i$, let $U_i$ be  an $\RC$-JSJ tree relative to $\calh_i^{+ab}$. 
There exists a number $q$, depending only on $G$, such that 
the trees $U_i$ belong to at most $q$ distinct deformation spaces. 
\end{prop}

\begin{proof} 
Let $U_i$ be as in the proposition. 
Note that $U_i$ satisfies condition 1 of Definition \ref{rcj} 
with respect to $\calh_j^{+ab}$ if $j\le i$, and condition 2 with respect to $\calh_j^{+ab}$ if $j\ge i$.  But   cyclic edge stabilizers of $U_i$ do not necessarily satisfy 2 with respect to $\calh_j^{+ab}$ if $j< i$. 

  In general, there is no domination map $U_i\ra U_{i+1}$, so we cannot apply Proposition \ref{dom2} directly.
The easy case is when, for each $i$,  every cyclic edge stabilizer  of $U_{i+1}$ 
is contained in  an edge stabilizer of $U_i$. 
Indeed, this implies that  $U_{i+1}$
satisfies condition 2   with respect  to $\calh_i^{+ab}$ (not just to $\calh_{i+1}^{+ab}$).
By condition 3, 
$U_i$ dominates $U_{i+1}$,
so Proposition \ref{dom2} applies. 

 Next,  assume that there is an $\RC$-tree $T$ relative to $\calh_1$ such that, for all $i$,
there is a domination map $T\ra U_i$ that collapses no edge.
Each  cyclic edge stabilizer $G_e$ of  $U_{i+1} $ contains an edge stabilizer $G_{e'}$ of $T$ (take for $e'$ any edge whose image contains a subarc of $e$). Since $G$ is freely indecomposable relative to $\calh_1$, and $T$ is relative to $\calh_1$, one has $G_{e'}\neq 1$, and  
$G_{e'}=G_e$ because $G_{e'}$ is root closed.
Since the map $T\ra U_{i}$ collapses no edge, $G_e$ fixes an edge in $U_i$,    and we conclude as above.

We now construct such a tree $T$.
By condition 2 of  Definition \ref{rcj}, edge stabilizers of $U_1$ are elliptic in $U_2$, 
so by Lemma \ref{ecl} there is  an $\RC$-tree $T_1$ relative to $\calh_1$ 
which refines $U_1$ and dominates $U_2$; we remove redundant vertices of $T_1$ if needed. 
Edge stabilizers of $T_1$ fix an edge in $U_1$ or  $U_2$, so are elliptic in $U_3$ and one may iterate. One obtains     $\RC$-trees  $T_i$  relative to $\calh_1$ such that $T_i$ refines $T_{i -1}$ and dominates $U_{i+1}$. By Proposition  \ref{access}, all trees $T_i$ for $i$ large enough are equal to a fixed    $\RC$-tree  $T$.
We have no  control over how large $i$ has to be, but we have a uniform bound for the number of orbits of edges of $T$. 

By construction, there are domination maps $f_i:T\to U_i$, but $f_i$ may collapse some  $G$-invariant set of edges.
There are only a bounded number of possibilities for the set $E_i$ of edges of $T$ that are collapsed by $f_i$,
so we  
may assume that $E=E_{i}$ is independent   of $i$. Collapsing all edges of $E$ then gives
a tree $T$ as wanted.
 \end{proof}

   \section{The chain condition}\label{pfcc}
 
 \renewcommand {\calc} {{\mathcal {H}}}

We prove  Theorem \ref{mccc}. In this section we only consider groups of the form $\Out(G;\mk\calh)$, so we use the simpler  notation $\M(\calc )$.
Since we do not yet know that every $\M(\calc )$ is a McCool group, we assume that every $\calc_i$ is a finite set of finitely generated subgroups (this is needed to apply Lemma \ref{lem_fini}).

Since   $\M(\calc')=\M(\calc\cup\calc')$ if $\M(\calc)\supset \M(\calc')$, we may assume $\calc_i\inc\calc_{i+1}$.
 We will use the following procedure several times.  We associate an invariant to each family $\calc_i$, and we show that, as $i$ varies, the number of distinct values  of the invariant is   bounded (by which we mean that there is a bound depending only on $G$). We then continue the proof under the additional asssumption that the value of the invariant is independent  of $i$.
 
 $\bullet$
 The first invariant is the Grushko deformation space $\cald_i$ relative to $\calc_i$  (see Subsection \ref{jsj}). 
 The assumption $\calc_i\inc\calc_{i+1}$ implies that $\cald_i$ dominates $\cald_{i+1}$. 
 As in the proof of Proposition \ref{dom2}, it follows from
 \cite{Gui_actions} that the number of times that $\cald_i$ changes is   bounded. We may therefore assume that $\cald_i$ is constant.

Let $G_1,\dots,G_n$ be the free factors in a Grushko decomposition $G=G_1*\dots*G_n*F_p$ relative to $\calc_i$
(they do not depend on $i$ up to conjugation since $\cald_i$ is constant).
The subgroup of $\M(\calc_i)$ consisting of automorphisms sending each factor $G_j$ to a conjugate has   bounded index,
 and it is determined by the McCool groups
  $\M_{G_j}(\calc_i{}_{ | G_j})$, so we are reduced to the case when $G$ is freely indecomposable relative to $\calc_i$.
 
   $\bullet$ 
  We then consider the canonical   JSJ tree $T_i$ (over abelian subgroups) relative to $\cip$, i.e.\  to $\calc_i$ and all non-cyclic abelian subgroups (see Subsection \ref{jsj}); it is $\M(\calc_i)$-invariant. We cannot use Proposition \ref{edstab} to say that the number of distinct $T_i$'s is bounded, because they are not $\RC$-trees,  so we shall now   replace $T_i$ 
  by an $\RC$-JSJ tree $U_i$.

   Any edge $e$ of $T_i$ joins a vertex ${v_1}$ whose    stabilizer is a maximal abelian subgroup to a vertex ${v_0}$ with non-abelian stabilizer. The group $G_e$ is a maximal abelian subgroup of  $G_{v_0}$, but not necessarily of $G_{v_1}$. Let $\bar G_e$ be the root-closure of $G_e$ in $G_{v_1}$ (hence also in $G$). As in Section 4.3 of \cite{DG2}, we  can fold all  edges in the $\bar G_e$-orbit of $e$ together. Doing this for all edges of $T_i$ yields an $\RC$-tree $U_i$ which is $\M(\calc_i)$-invariant. 
 
 This construction may also be described in terms of graphs of groups, as follows. We now view 
 $e=v_0v_1$ as an edge of $T_i/G$. Subdivide it by adding a midpoint $u$ carrying $\bar G_e$. This creates two edges $v_0u$ and $uv_1$, carrying $G_e$ and $\bar G_e$ respectively. Do this for every edge $e$ of $T_i/G$. Collapsing all edges $uv_1$ yields $T_i/G$, whereas collapsing all edges $v_0u$ yields $U_i/G$.

The quotient graph $U_i/G$ is the same as $T_i/G$, but labels are different. Edge groups are replaced by their root-closure, and  non-abelian vertex groups have gotten bigger  (roots have been adjoined:   each fold replaces some
$G_{v_0}$  by $G_{v_0}*_{G_e}\bar G_e$). Just like $T_i$, the tree $U_i$ is equal to its tree of cylinders because folding only occurs within cylinders; in particular, $U_i$ is determined by its deformation space.

 Note that $U_i$ may have redundant vertices, and is not necessarily minimal (this happens if $T_i/G$ has a terminal vertex carrying an abelian group, and the incident edge group has finite index). In this case we replace $U_i$ by   its minimal subtree. 
 
 We claim that $U_i$ is an  $\RC$-JSJ tree relative to $\calh_i^{+ab}$, in the sense of Definition \ref{rcj}. It satisfies conditions 1 and 2 since its edge stabilizers are finite extensions of edge stabilizers of $T_i$.
 Any tree  satisfying these two conditions is dominated by $T_i$ because $T_i$ is a JSJ tree. But any $\RC$-tree dominated by $T_i$ is also dominated by $U_i$ (with notations as above, $e$ and $ge$ must have the same image if $g\in \bar G_e$).

  $\bullet$
 Proposition \ref{edstab} lets us  assume that $U_i$ is a fixed tree $U$.   It is invariant under every $\M(\calc_i)$.  
 We let $\Out^0(U)$ be  the finite index subgroup of $\Out(U)$ consisting of automorphisms preserving $U$ and acting trivially on  $\Gamma=U/G$.
   The number of edges of $\Gamma$ is uniformly bounded by Proposition \ref{access}, so the index of $\Out^0(U)$ in $\Out(U)$
  is bounded, and  it is enough to prove the chain condition for $\M^0(\calc_i):=\M(\calc_i)\cap\Out^0(U)$. 

Let $V$ be the set of vertices of $\Gamma $. 
As recalled in Subsection \ref{automs}, there are maps $\rho_v:\Out^0(U)\to \Out(G_v)$ and a product   map $\rho:\Out^0(U)\to\prod_{v\in V}\Out(G_v)$. Since $U$ is relative to $\calc_i$, the group of twists $\Tw=\ker\rho$ is contained in $\M^0(\calc_i)$. 

\begin{lem} \label{outun}
There exist subgroups $\Out^1(G_v)\inc \Out(G_v)$, independent of $i $, such that:
\begin{enumerate}
\item $\prod_{v\in V}
\Out^1(G_v)$ is contained in $\rho(\M^0(\calc_i))$ for every $i$;
\item the index of $\Out^1(G_v)$ in $\rho_v(\M^0(\calc_i))$ is uniformly bounded.
\end{enumerate}
\end{lem}

This lemma implies Theorem \ref{mccc} because $\M^0(\calc_i)$ 
contains $\rho\m(\prod_{v\in V}
\Out^1(G_v))$ with bounded index.

\begin{proof}[Proof of Lemma \ref{outun}]
 
Let $\calc_{i,v}:=(\calc_{i})_{||G_v}$  be 
the set of (conjugacy classes of) subgroups of $G_v$ which are conjugate to an element of $\calc_i$, and which fix no other point in $T$ (see Subsection \ref{tre}).
Since two such subgroups are conjugate in $G_v$ if and only if they are conjugate in $G$,
we may view $\calc_{i,v}$ as a subset of $\calc_i$.

Since $\rho (\M^0(\calc_i))$ contains $\prod_{v\in V}\M(\Inc_v\cup\calc_{i,v} )$,  as explained in Subsection \ref{automs},
it suffices 
to fix $v\in V$ and to construct $\Out^1(G_v)$, \emph{with $\Out^1(G_v)\inc   \M(\Inc_v\cup\calc_{i,v})$  and the index of $\Out^1(G_v)$ in $\rho_v(\M^0(\calc_i))$   uniformly bounded.} We distinguish several cases.

$\bullet$ First suppose that $G_v\simeq \Z^k$ is abelian, so $\Out(G_v)=\Aut(G_v)=GL(k,\Z)$. 
 Let $A_i$ be  the root-closure of the subgroup of $G_v$ generated by incident edge groups and subgroups  in $\calc_{i,v} $. It is a direct factor and increases with $i$, so we may assume that it is independent of $i$. We define $\Out^1(G_v)\inc \Out(G_v)$ as the subgroup consisting of automorphisms equal to the identity on $A_i$. It is  equal to $\M(\Inc_v\cup\calc_{i,v})$ and contained in $\rho_v(\M^0(\calc_i))$. We must show that the index is bounded.

  The group $A_i$ is invariant under $\rho(\M^0(\calc_i))$, and we have to bound the order of the image of $\M^0(\calc_i)$ in $\Out(A_i)$. Any incident edge group $\bar G_e$ of $G_v$ contains an edge stabilizer $G_e$ of $T_i$ with finite index, and 
 the image of the map $\rho_e: \M^0(\calc_i)\to \Out(G_e)$ is finite  by Lemma \ref{lem_fini}. Since $A_i$ is generated by incident edge groups and  elements which are fixed by $\M^0(\calc_i)$, this implies that the image of $\M^0(\calc_i)$ in $\Out(A_i)$ is finite. Its cardinality is uniformly bounded because there is a bound for the order of finite subgroups of $GL(k,\Z)$, so the index of $\Out^1(G_v)$ in $\rho_v(\M^0(\calc_i))$ is bounded.

$\bullet$ 
We now consider a non-abelian vertex stabilizer $G_v$. 
It follows from the way   $U_i$ was constructed that 
$G_v$ is, for each $i$, the fundamental group of a graph of groups $\Lambda_{i,v}$. This graph      is a tree. It has a   central vertex $v_i$, which may be viewed as a vertex of $
T_i/G$ with $G_{v_i}$ non-abelian. All edges $e$ join $v_i$ to a vertex $u_e$ carrying a root-closed abelian group, and the index of $G_e$ in $G_{u_e}$ is finite. 
The graph of groups $\Lambda_{i,v}$ is invariant under the action of  $ \M^0(\calc_i)$ on $G_v$.

We say that $G_v$ (or $v$) is  \emph{rigid with sockets},  or \emph{QH with sockets}, depending on the type of $v_i$ as a vertex of $T_i$ (since the number of vertices of $T_i/G$ is bounded, we may assume that this type  is independent of $i$).

$\bullet$ If $G_v$ is rigid with sockets, we define 
  $\Out^1(G_v)$ as the trivial group, and we have to explain why $\rho_v(\M^0(\calc_i))$ is a finite group of bounded order.  
Assume first that  $U=T_i$ (i.e.\ $U $ is also a regular JSJ tree).
  Lemma \ref{lem_fini} then implies
that $\rho_v(\M^0(\calc_i))$ is a finite  subgroup of $G_v$, 
 but we  need to bound its order only in terms of $G$ (independently of the sequence $\calc_i$).
To get this uniform bound, we note that there are only finitely many possibilities for $G_{v}$ up to isomorphism
by \cite{GL_vertex}. Moreover  $\Out(G_{v})$ is virtually torsion-free by \cite[Cor 4.5]{GL6}, so there is a bound for the order of its finite subgroups. 

In general (\ie without assuming $U=T_i$),
we study $\rho_v(\M^0(\calc_i))$ through its action on the graph of groups $\Lambda_{i,v}$ as in Subsection \ref{automs} (note that edges are not permuted). 
  The group of twists is trivial because edge groups are maximal abelian in $G_{v_i}$  and   terminal vertex groups are abelian (see Proposition 3.1 of \cite{Lev_automorphisms}), so we only have to control the action of $ \M^0(\calc_i) $ on vertex groups of $\Lambda_{i,v}$. 
  
Applying Lemma \ref{lem_fini} to the JSJ decomposition $T_i$,
we get  finiteness of  the image of $ \M^0(\calc_i )$ in $\Out(G_{v_i})$, 
and in $\Out(G_e)$  for every edge $e$ of $T_i$, and hence of $\Lambda_{i,v}$.
The action of an automorphism on the edge groups of $\Lambda_{i,v}$ determines the action on the abelian vertex groups 
because they contain the incident edge group with finite index. 
This proves that $\rho_v(\M^0(\calc_i))$ is finite, and boundedness follows as above.

$\bullet$ There remains the case when $G_v$ is QH with sockets. The group 
$G_{v_i}$   is then isomorphic to the fundamental group of a compact surface $\Sigma_i$, and incident edge groups are boundary subgroups.  The topology of $\Sigma_i$ may vary with $i$, but the number of boundary components of $\Sigma_i$ is bounded (by a simple   accessibility argument, or because the rank of $G_{v_i}$  as a free group  is bounded  by \cite{GL_vertex}).

 If  $J$ is a subgroup of $G$, denote by $\calu_i(J)$ the set of elements of $J$ that are 
 $\cip$-universally elliptic (\ie elliptic in every   $G$-tree with abelian edge stabilizers which is relative to $\calc_i$ and non-cyclic abelian subgroups). We view it as a union of $J$-conjugacy classes. Since $\calc_i\inc\calc_{i+1}$, we have $\calu_i(J)\inc\calu_{i+1}(J)$. We shall show that the sequence $\calu_i(G_v)$ stabilizes.

 We first study   $\calu_i(G_{v_i})$: we   claim that $\calu_i(G_{v_i})$ is the 
union of the conjugacy classes 
of boundary subgroups of $G_{v_i}=\pi_1(\Sigma_i)$.
Indeed, any boundary subgroup is an incident edge group of $v_i$ (up to conjugacy) 
or 
has 
a finite index subgroup conjugate to a group  in $\calc_i$ 
(otherwise,  $G$ would be freely decomposable relative to $\calc_i$, see Proposition 7.5 of \cite{GL3a}). 
It follows that $\calu_i(G_{v_i})$ contains all boundary subgroups
(incident edge groups are $\cip$-universally elliptic because $T_i$ is a JSJ tree relative to $\cip$).   
Conversely, by Proposition 7.6 of \cite{GL3a}, any $g\in\calu_i(G_{v_i})$ is contained in a boundary subgroup of $\pi_1(\Sigma_i)$.  This proves our claim, and shows in particular that $\calu_i(G_{v_i})$ is the union of a bounded number of conjugacy classes of maximal cyclic subgroups $L_j(i)$ of $G_{v_i}$.

  We now consider $\calu_i(G_{v})$. 
The $\cip$-universally elliptic elements of $G_v$ are contained (up to conjugacy) in $G_{v_i}$ or in one of the  terminal vertex groups of $\Lambda_{i,v}$, so  $\calu_i(G_v)$ is the union of the conjugates of the root-closures (in $G_v$) of the groups $L_j(i)$.  
  Since $\calc_i\inc\calc_{i+1}$, we have $\calu_i(G_v)\inc\calu_{i+1}(G_v)$. As $\calu_i(G_v)$ is the union of the conjugates of a bounded number of cyclic subgroups, 
  we may assume that $\calu_i(G_v)=\calu(G_v)$ does not depend on $i$. 
  
  Elements of $\rho_v(\M^0(\calc_i))$ send each cyclic group in $\calu(G_v)$ to a conjugate (conjugacy classes are not permuted because the action on $T_i/G$ is trivial). They act trivially on 
   groups  in $\calc_{i,v}$, but they may  map an element $g$ belonging to a terminal vertex group of $\Lambda_{v,i}$
   to $g\m$ (geometrically, they correspond  to homeomorphisms of $\Sigma_i$ which may reverse orientation on boundary components). 
  
  We   define $\Out^1(G_v)\inc\Out (G_v)$ as the group of automorphisms  acting trivially on each cyclic group in $\calu(G_v)$ 
  (geometrically, we restrict to  homeomorphisms of $\Sigma_i$   equal to the identity on the boundary).   It is contained in $\M(\Inc_v\cup\calc_{i,v})$, because  $\calu_i(G_v)$ contains the incident edge groups of $G_v$ in $U$, 
hence  contained in  $\rho_v(\M^0(\calc_i))$, and the index is bounded in terms of the number of conjugacy classes of cyclic subgroups in $\calu(G_v)$.
 \end{proof}

\begin{rem} \label{ccab}

Groups of the form $\Out(G;\calh)$, with $\calh$ a finite family of abelian groups, do not satisfy the descending   chain condition: consider $G=\Z^2=\langle x,y\rangle$, and $\calh_i=\{\langle x,y^{2^i}\rangle\}$.
\end{rem}

   \section{Proof of the other results
   } \label{pfcor}
   
 \renewcommand {\calc} {{\mathcal {C}}} 
 
   We first note the following consequence of the chain condition:
   
   \begin{prop} \label{cinf}
   If $\calc$ is an infinite family of conjugacy classes, there exists a finite subfamily $\calc'\inc\calc$ such that $\M(\calc)=\M(\calc')$.
   \end{prop}
   
    Recall that $\M(\calc)$ is the group of outer automorphisms fixing all conjugacy classes belonging to $\calc$.
   
  \begin{proof} Write $\calc$ as an increasing union of finite families $\calc_i$ and note that $\M(\calc)$ is the intersection of the descending chain $\M(\calc_i)$.
  \end{proof}

To prove  Corollary \ref{genmc}, saying in particular that   every   McCool group
 is an elementary McCool group, 
we need the following fact:
 
 \begin{lem} \label {minos}
 Let $G$ be a toral relatively hyperbolic group. Let $H$ be a subgroup, and  $\alpha\in\Aut(G)$. If $\alpha(h)$ and $h$ are conjugate   in $G$ for every   $h\in H$, then $\alpha$ acts on $H$ as conjugation by some $g\in G$. 
 \end{lem}
 
  \begin{proof}  We may assume that there is a non-trivial $h\in H$ such that $\alpha(h)=h$. If $H$ is 
 abelian, malnormality   of maximal abelian subgroups implies that $\alpha$ is the identity on $H$. If not, the result follows from Lemma 5.2 of  \cite{MiOs_normal} (which is valid for any homomorphism $\varphi:H\to G$, not just automorphisms of $H$), see also Corollary 7.4 of 
\cite{AMS_commensurating}.

  \end{proof}

\begin{cor*}[Corollary \ref{genmc}]
Let $G$ be a toral relatively hyperbolic group. 
If  $\calh $ is any 
family of 
subgroups of $G$, 
there exists a finite set of conjugacy classes 
such that $\M(\calh)=\M(\calc)$. 
\end{cor*}

Recall that  $\M(\calh)$ is also denoted $    \Out(G;\mk\calh)$. We favor the notation  $\M(\calh)$   in   this section.

   \begin{proof}
    Given   an arbitrary family $\calh$, let $\calc_\calh$ be the set of all conjugacy classes having a representative belonging to some $H_i$. By Lemma \ref{minos}, $
    \M(\calh)=\M(\calc_\calh)$. We apply 
Proposition \ref   {cinf} to get $
  \M(\calh)=\M(\calc)$ with $\calc$ finite.
  \end{proof}

Together with Theorem \ref{mcgg},  this  implies our most general finiteness result.
 \begin{cor} \label{thm_general}
 Let $G$ be a toral relatively hyperbolic group.
Let $\calh$ be an arbitrary collection of   subgroups of $G$. 
Let $\calk$ be a  finite 
collection of abelian subgroups of $G$.
 Let $T$ be a simplicial tree on which $G$ acts with abelian edge stabilizers, 
 with each group in $\calh\cup \calk$ fixing a point. 

 Then  the group  $\Out(T,\mk\calh,\calk)=\Out(T)\cap\Out(G;\mk\calh,\calk)$ of automorphisms leaving $T$ invariant,
acting trivially on each group of $\calh$, and sending each $K\in \calk$ to a conjugate (in an arbitrary way),
is of type VF. 
 \end{cor}

  \begin{proof}
 By Corollary \ref{genmc}, we may write $\Out(G;\mk\calh)=\M(\calc)$ for some finite family of conjugacy classes $[c_i]$, with each $c_i$ belonging to a group of $\calh$ hence elliptic in $T$. Defining $\call=\{\langle c_i\rangle\}$, we see that $\M(\calc)$ is a finite index subgroup of $ \Out(G;\call)$, so $\Out(T,\mk\calh,\calk)$ is a finite index subgroup of $\Out(T,\calk\cup\call)$. 
By Theorem \ref{mcgg}, this group has type VF, 
and therefore so does $\Out(T,\mk\calh,\calk)$. 
  \end{proof}

   Proposition  \ref{infmc} and Theorem \ref{uccfix} will be proved at the end of the section.

\begin{prop*}[Proposition \ref{mccet}] 
Given  a toral relatively hyperbolic group $G$, there exists a number $C$ such that,  
if a subgroup $\wh M\subset \Out(G)$  contains a  group $  \M(\calh)$ with finite index, then 
 the index $[\wh M:\M(\calh)]$ is bounded by $C$.
\end{prop*}

   \begin{proof} [Proof of Proposition \ref{mccet}]  
   
 By   Corollary \ref{genmc}, we may write 
$ \M(\calh)=\M(\calc')$ for some finite set $\calc'$.
Let $\calc$ be the orbit of $\calc'$ under $\wh M$. Since $\M(\calc')$ fixes $\calc'$, 
this is a finite $\wh M$-invariant collection of conjugacy classes. 
We thus have  $$\M(\calc)\subset \M(\calc')\subset \wh M\subset \wh\M(\calc),$$
and it suffices to bound the index $[\wh\M(\calc):\M(\calc)]$.

 As in  the beginning of  Section  \ref{pfcc}, 
    let $G=G_1*\dots *G_n*F_r$ be a Grushko decomposition of $G$ relative to $\calc$, and $\calg=\{G_1,\dots G_n\}$.
 The group  $\wh\M(\calc)$ permutes the conjugacy classes of the groups in $\calg$. 
Since the cardinality of $\calg$ is bounded, and $G$ has finitely many free factors up to isomorphism, we may assume that $G$ is one-ended relative to $\calc$. 
   
We now consider the   JSJ decomposition $\Tcan$ over abelian groups relative to $\calc$ and non-cyclic abelian groups. It is invariant under $\wh\M(\calc)$, so we may study $\wh\M(\calc)$ through its action on $\Tcan$ (see Subsection \ref{automs}). 

The number of edges of $\Gcan=\Tcan/G$ being bounded by  the first case of Proposition \ref{access}, we may replace $\wh\M(\calc)$ and $\M(\calc)$ by  their subgroups $\wh\M^0(\calc)$ and $\M^0(\calc)$  acting trivially  on $\Gamma$. The group of twists $\calt$ is contained in  $\M^0(\calc)$, so 
as in the proof of Lemma \ref{outun} it suffices to construct $\Out^1(G_v)\inc\M_{G_v}(\Inc_v\cup\calc_{ |  | G_v} )$  with the index of $\Out^1(G_v)$ in $\rho_v(\wh\M^0(\calc))$ uniformly bounded. 
We distinguish the same cases  as in 
 the proof of Lemma \ref{outun}.

  If $G_v$ is abelian, isomorphic to $\Z^k$ with $k\ge2$, let $H<G_v$ be
the set of elements whose orbit under $\rho_v(\wh\M^0(\calc))$ is finite. This is 
a subgroup of $G_v$, isomorphic to some $\Z^p$, which is invariant under $\rho_v(\wh\M^0(\calc))$
and contains the  incident edge groups by Lemma \ref{lem_fini}. We define $\Out^1(G_v)=\M_{G_v}(\{H\})$. It is contained in $\M_{G_v}(\Inc_v\cup\calc_{ |  | G_v} )$.
The image  of  $\rho_v(\wh\M^0(\calc))$ in $\Aut(H)=GL(p,\Z)$ is  finite, and its order bounds the index of $\Out^1(G_v)$  in  $\rho_v(\wh\M^0(\calc))$. 
 This concludes the proof in this case since there is a bound for the order of finite subgroups of $GL(p,\Z)$.

If $G_v$ is rigid, we let $\Out^1(G_v)$ be trivial. The image of $\wh\M^0(\calc)$ in $\Out(G_v)$ is finite by Lemma \ref{lem_fini}, and bounded by \cite{GL_vertex} as in the proof   of Lemma \ref{outun}. 

If $G_v=\pi_1(\Sigma)$ is QH, we define $\Out^1(G_v)=\calp\calm^+(\Sigma)=\M_{G_v}(\Inc_v \cup\calc_{||G_v} )$. Elements of $\rho_v(\wh\M^0(\calc))$ may reverse orientation, or permute boundary components of $\Sigma$.
 \end{proof}

\begin{cor}\label{eucc}
Extended   elementary McCool groups  $\wh\M(\calc)$ of $G$ satisfy a uniform chain condition. 
\end{cor} 

\begin{proof}
Given a descending chain $\wh\M(\calc_i)$, define $\calc'_i=\calc_0\cup\dots\cup \calc_i$ and note that 
$$\M(\calc'_ i)=\cap_{j\le i} \M(\calc_j)\inc \wh\M(\calc_i)=\cap_{j\le i}\wh\M(\calc_j)\inc \wh\M(\calc'_i).$$
The corollary follows from Theorem \ref{mccc}, since by Proposition \ref{mccet} the index of $\M(\calc'_ i)$ in $\wh\M(\calc'_i)$ is bounded.
\end{proof}

We now prove Corollary \ref{rless} stating that, for any $A<\Out(G)$, there is a subgroup $A_0<A$  of bounded finite index such that,  for the action of $A_0$ on the set of conjugacy classes of $G$,  every  orbit is   a singleton  or is infinite.

\begin{proof}  [Proof of Corollary \ref{rless}]
Let $\calc_A$ be the (possibly infinite) set of  conjugacy classes of $G$ whose $A$-orbit is finite.
Partition $\calc_A$ into $A$-orbits, and let  $\calc_p$ be the union of the first $p$ orbits. The image of $A$ in the group of permutations of $\calc_p$ is contained in that of $\wh\M(\calc_p)$, so  by Proposition \ref{mccet} its order is bounded by some fixed $C$. This $C$ also bounds the order of the image of  $A$ in the group of permutations of $\calc_A$.
\end{proof}

Recall that   $
\AM(\calh, H_0)\inc\Aut(G)$ is the group of automorphisms acting trivially on $\calh$ (in the sense of Definition \ref{gmc},  i.e.\ by conjugation) and fixing the elements of $H_0$.
Proposition \ref{mcaut} 
states that, if $G$ is non-abelian,  then $
 \AM(\calh, H_0)$ is an extension $$1\to K\to
 \AM(\calh, H_0)\to \M(\calh')\to 1$$ with $\M(\calh')\inc\Out(G)$   a McCool group, and   $K $   the centralizer of $H_0$. 
Corollary \ref{mca}
states that the groups $ \AM(\calh, H_0)$ are of type VF 
and satisfy a uniform chain condition.

   \begin{proof} [Proof of Proposition \ref{mcaut}] 
   Let $\calh'=\calh\cup\{H_0\}$. 
 Map 
 $ 
 \AM(\calh, H_0)\subset \Aut(G)$
  to $\Out(G)$. The image is $\M(\calh')$. 
The kernel  $K$ is the set of inner automorphisms equal to the identity on $H_0$. Since $G$ has trivial center, it is isomorphic to the centralizer of $H_0$. 
   \end{proof}

    \begin{proof} [Proof of Corollary \ref{mca}] 
     The group $\M(\calh')$ has type VF by Theorem  \ref{mc}. The group $K$ is abelian or equal to $G$, so   has type F  because    $G$ does   \cite{Dah_classifying}. Proposition \ref{mcaut} and Corollary \ref{cor_extensionVF} imply that 
    $
 \AM(\calh, H_0)$
has type VF. 
Moreover, a chain of centralizers has length at most 2 since the centralizer of $H_0$ is trivial, $G$, or a maximal abelian subgroup. The uniform chain condition for McCool groups (Proposition \ref{mccc}) then implies the uniform chain condition for groups of the form $
 \AM(\calh, H_0)$.
  \end{proof}

We now deduce the bounded chain condition for fixed subgroups.

  \begin{proof} [Proof of Theorem \ref{uccfix}]  Let  $J_0\incs J_1\incs\dots\incs J_p$ be a strictly ascending chain of fixed subgroups.
Let $\AM(\es,J_i)$ be  the subgroup of $\Aut(G)$ consisting of automorphisms equal to the identity on $J_i$. 
Since $J_i$ is a fixed subgroup,  $\AM(\es, J_i)\supsetneq \AM(\es, J_{i+1})$.
Corollary \ref{mca} then gives a bound on the length on the chain.
\end{proof}

\begin{rem*} One can adapt the arguments of Section \ref{pfcc} to prove Theorem \ref{uccfix} directly (without passing through McCool groups).
\end{rem*}

We now prove   Proposition \ref{infmc} saying that $\Out(F_n)$ contains infinitely many non-isomorphic McCool groups
for $n\geq 4$, and infinitely many non-conjugate McCool groups for $n\geq 3$.

\begin{proof}[Proof of Proposition \ref{infmc}]
Let $H$ be the free group on three generators $a,b,c$. Given a non-trivial element $w\in\langle a,b\rangle$, let $P_w$ be the cyclic HNN extension $P_w=\langle a,b,c,t\mid tct\m=w\rangle$. It is free of rank 3, with basis $a,b,t$. Let $\varphi_w$ be the automorphism of $P_w$ fixing  $a$ and $b$,
and mapping $t$ to $wt$    
(it  equals  the identity on $H$ since it fixes $c=t\m wt$). 
The image   $\Phi_w$ of $\varphi_w$ in $\Out(P_w)$  preserves the Bass-Serre tree $T$ of the  HNN extension (it belongs to   its group of twists $\calt$).

We   apply this construction  with $w=a^kb^k$, for $k$ a positive integer. As $k$ varies, the 
  cyclic subgroups $\langle\Phi_w\rangle$
are pairwise non-conjugate in $\Out(P_w)\simeq\Out(F_3)$, as seen by considering the action on the 
abelianization.

We shall now prove  the second assertion of the proposition for $n=3$,
 by showing that
  $\langle\Phi_w\rangle$ is a   McCool group of $P_w$,   
namely  $\langle\Phi_w\rangle=\M_{P_w}(\{H\})\subset \Out(F_3)$. 
The extension to $n>3$ is straightforward, by adding generators to $H$.

Consider splittings of $P_w$ over abelian (i.e.\ cyclic) subgroups relative to $H$.
The  tree $T$ 
is a JSJ tree because its vertex stabilizers are universally elliptic \cite[lemma 4.7]{GL3a}; in particular, $P_w$ is freely indecomposable relative to $H$. Moreover, $T$ equals its tree of cylinders (up to adding redundant vertices) because $w$ is not a proper power, so $T$ is the canonical JSJ tree $\Tcan$. The McCool group  $\M_{P_w}(\{H\})$ therefore  leaves $T$ invariant, and it is easily checked using \cite{Lev_automorphisms} that 
$\M_{P_w}(\{H\})=\Tw= \langle\Phi_w\rangle$.

To prove the first assertion  of the proposition,  consider $R_w=P_w*\langle d\rangle\simeq F_4$,  the family $\calh=\{H,\langle d\rangle\}$, and the McCool group
$\M_{R_w}(\calh)\subset \Out(F_4)$.
The decomposition $R_w=P_w*\langle d\rangle$ is a Grushko decomposition of $R_w$ relative to 
 $\calh$ because $P_w$ is freely indecomposable relative to $H$.
This decomposition   is invariant under $\M_{R_w}(\calh)$ 
because it is a one-edge splitting (see \cite{For_deformation}, cor 1.3).

The stabilizer    $\Out(T)$ of the Bass-Serre tree $T$ in $\Out(R_w)$ is naturally isomorphic to 
$$\Aut(P_w)\times \Aut(\grp{d})\simeq \Aut(P_w)\times\Z/2\Z$$ 
(see \cite{Lev_automorphisms});
  the natural map $\Out(T)\ra \Out(P_w)$ kills the factor $\Z/2\Z$ and 
coincides with the quotient map $\Aut(P_w) \ra \Out(P_w)$ on the other factor.
The McCool group $\M_{R_w}(\calh)$ is isomorphic to
 the preimage   of $\M_{P_w}(\{H\})=\grp{\Phi_w}$  in $\Aut(P_w)$,  
 hence to the  mapping torus   
$$Q_w= \langle a,b,t,u\mid ua=au, ub=bu, utu\m= a^kb^kt \rangle.$$
 The  abelianization of $Q_w$ is $\Z^3\times \Z/k\Z$, so  the isomorphism type of $Q_w$ changes when $k$ varies. This proves the first assertion of the proposition  or $n=4$. The extension to larger $n$ is again straightforward.
\end{proof}

\section{Appendix: groups with finitely many McCool groups
}

In this appendix we 
describe cases when $\Out(G)$ only contains finitely many McCool subgroups. In particular, we
show that the values of $n$ given in  Proposition \ref{infmc} are optimal.

\begin{prop} \label{mccoolfini} If $G$ is a torsion-free one-ended hyperbolic group, then $\Out(G)$ only contains finitely many McCool groups up to conjugacy. 
\end{prop}

\begin{prop} \label{infmc2}
$\Out(F_2)$ only contains finitely    many   McCool groups up to conjugacy. 
\end{prop}

\begin{prop} \label{infmc3} $\Out(F_3)$ only contains finitely many   McCool groups up to isomorphism.
\end{prop}

The proof of Proposition
\ref{mccoolfini} requires the fact that $\Out(G)$, and more generally extended McCool groups $\widehat\M(\calc)$, only contain finitely many conjugacy classes of finite subgroups. This will appear in \cite{GL_extension}.

 \begin{proof}[Proof of Proposition
\ref{mccoolfini}]
We assume that $\Out(G)$ contains infinitely many non-conjugate elementary McCool groups $\M(\calc_i)$, and we derive a contradiction
(this   implies the proposition by   Corollary \ref{genmc}).

It is proved in   \cite[Corollary 4.9]{Sela_acylindrical} 
 that there are only finitely many minimal actions of $G$ on trees with cyclic edge stabilizers, up to the action of $\Out(G)$, so we may assume
that the canonical cyclic JSJ tree relative to $\calc_i$ (the tree $\Tcan$ of  Subsection \ref{jsj}) is a given tree $T$. This tree   is invariant under all groups $\M(\calc_i)$, so $\M(\calc_i)\inc \Out(T)$. In this proof, we cannot restrict to  $\Out^0(T)$.

Given a vertex $v$ of $T$, we define $\calc_{i,v}$ 
as the restriction $\calc_i{}_{ | G_v}$ if $G_v$ is cyclic,
as $\calc_i{}_{ |  | G_v}$ if $G_v$ is not cyclic (recall from Subsection \ref{tre} that conjugacy classes  represented by elements fixing an edge of $T$ do not belong to   $\calc_i{}_{ || G_v}$). The tree being bipartite, $\calc_i$ is the disjoint union of the $\calc_{i,v}$'s.

We say that $v$ is used if 
$\calc_{i,v}$ is non-empty. Since there are finitely many $G$-orbits of vertices, we may assume that usedness is independent of $i$; we let $V_u$ be a set of representatives of  orbits of used vertices. 
We may also assume that the type of vertices with non-cyclic stabilizer  (rigid or   QH) is independent of $i$ (QH vertices with $\Sigma$ a pair of pants are rigid, we do not consider them as QH).

We claim that QH vertices $G_v$ of $T$ are not used. Indeed, 
any boundary subgroup of $G_v$ is an incident edge stabilizer 
of $T$: otherwise, $G_v$ would split as a free product relative to $\Inc_v$, contradicting  one-endedness of $G$. Elements in $\calc_i$ are universally elliptic (relative to $\calc_i$), and the only universally elliptic subgroups of $G_v$ are contained in boundary subgroups of $G_v$ because $G_v$ is flexible (see \cite{GL3a}, Proposition 7.6), 
so $\calc_{i||G_v}$ is empty.

For $v\in V_u$, define $\Out_i(G_v)\inc\Out(G_v)$ as the set of automorphisms which fix each conjugacy class in  $\calc_{i,v}$  
and leave the set of incident edge stabilizers globally invariant. 
Any automorphism in $\M(\calc_i)$ 
is an automorphism of $T$ which leaves   $G_v$ invariant (up to conjugacy), 
and   induces an automorphism belonging to $\Out_i(G_v)$. 
Conversely, any automorphism of $T$ satisfying these properties for every $v\in V_u$ lies in $\M(\calc_i)$.
This means that
 $\M(\calc_i)$ is completely determined by the knowledge of the groups $\Out_i(G_v)$, for $v\in V_u$.

We complete the proof by 
showing that there are only finitely many possibilities for each $\Out_i(G_v)$.
This is clear if $G_v$ is cyclic, and   QH vertices are not used, so there 
remains to consider the case where $G_v$ is rigid.

In this case,  $\Out_i(G_v)$ is finite by Lemma \ref{lem_fini} (otherwise  $G_v$ would have a cyclic splitting relative to $\Inc_v$ and $\calc_{i,v}$, contradicting rigidity).
Since $G_v$ is hyperbolic, $\Out(G_v)$ has finitely many conjugacy classes of finite subgroups \cite{GL_extension}. 
  We deduce that there are finitely many possibilities for $\Out_i(G_v)$, up to conjugacy in $\Out(G_v)$. Unfortunately, this is not enough to get finiteness for $\M(\calc_i)$ up to conjugacy in $\Out(G)$, because the conjugator may fail to extend to an automorphism of $G$.

To remedy this, we consider $\M(\Inc_v)$ and $\wh\M(\Inc_v)$, with $\Inc_v$ the family of incident edge groups as in Subsection \ref{tre}, and $\wh\M(\Inc_v)=\wh\Out(G_v;\Inc_v)$ the set of outer automorphisms of $G_v$ preserving $\Inc_v$ (see Definition \ref{pres}; edge groups may be permuted, and the generator of an edge group may be mapped to its inverse).

 The group $\Out_i(G_v)\inc\Out(G_v)$ is   finite and contained in 
 $\wh\M(\Inc_v)$ (but not necessarily in $\M(\Inc_v)$). 
 By \cite{GL_extension}, $\wh\M(\Inc_v)$ has only finitely many conjugacy classes of finite subgroups. 
It follows that there are only finitely many possibilities for $\Out_i(G_v)$ up to conjugation by an element of $\wh\M(\Inc_v)$,
 hence also up to conjugation by an element of   $\M(\Inc_v)$
since $\M(\Inc_v)$ has finite index in  $\wh\M(\Inc_v)$.

We may therefore assume that   
 $\Out_i(G_v)$ is independent of $i$ if  $G_v$ is cyclic and $v\in V_u$, and that all groups $\Out_i(G_v)$  are conjugate by elements of $\M(\Inc_v)$ if $v\in V_u$ is rigid. 
Any element of  $\M(\Inc_v)$ extends ``by the identity'' to an automorphism of $G$ which leaves $T$ invariant and acts trivially 
(as   conjugation by an element of $G$)  on $G_w$ if $w$ is not in the orbit of $v$. 
 Since 
 $\M(\calc_i)$ is determined by the groups $\Out_i(G_v)$ for $v\in V_u$, we conclude that all groups $\M(\calc_i)$ are conjugate in $\Out(G)$.
 \end{proof}

 \renewcommand {\calc} {{\mathcal {H}}}

\begin{proof}[Proof of Proposition \ref{infmc2}]
We view $\Out(F_2)\simeq GL(2,\Z)$ as the mapping class group of a punctured torus $\Sigma$ (with orientation-reversing maps allowed). Let $c$ be a peripheral conjugacy class (representing the commutator of basis elements of $F_2$). 

We consider a McCool group $\M(\calc)\inc\Out(F_2)$. We may assume that $\M(\calc)$ is infinite. 
By the classification of elements of $GL(2,\Z)$, or by the Bestvina-Paulin method and Rips theory, $F_2$ then splits over a cyclic group relative to $\calc$ and $c$ (see for instance Theorem 3.9 of \cite{GL6}). Such a  splitting is dual to a non-peripheral simple closed curve $\gamma \inc \Sigma$. 

If there are two different splittings, they are dual to curves $\gamma,\gamma'$ 
whose union fills $\Sigma$, so $\calc$ only contains peripheral subgroups.
It follows that $\M(\calc)$ is either $\Out(F_2)\simeq GL(2,\Z)$ or $SL(2,\Z)$.
If the splitting is unique, $\M(\calc)$ fixes $\gamma$ (viewed as an unoriented curve  up to isotopy). 
 Since the splitting dual to $\gamma$ is relative to $\calc$,   the Dehn twist $T_\gamma$ around $\gamma$ is contained in 
$\M(\calc)$. 
The stabilizer $\Stab(\gamma)$ of $\gamma$ in the mapping class group 
contains
    $\langle T_\gamma\rangle$ with finite index (the index is 4 because a
    homeomorphism may reverse   the orientation of $\Sigma$ and/or   
    of $\gamma$). 
We thus have  $\langle T_\gamma\rangle\inc \M(\calc)\inc \Stab(\gamma)$, with both indices finite.
Finiteness of $\M(\calc)$ up to conjugacy follows, since $\gamma$ is unique up to the action of the mapping class group.
\end{proof}

The remainder of this appendix is devoted to the proof of Proposition \ref{infmc3}.
 We first record a few useful facts.

\begin{lem}\label{vr}
Fix $n$.
Up to isomorphism,  $\Out(F_n)$ only contains finitely many virtually solvable subgroups.
\end{lem}

\begin{proof}
 Virtually solvable subgroups are virtually abelian (\cite{Ali_translation,BFH_solvable}). More precisely, they contain $\Z^k$ with $k\le2n-3$  as a subgroup of   bounded index (see \cite{BFH_solvable}, proof of Theorem 1.1 page 94). This implies finiteness, for instance by \cite[Th.\ 6 p.\ 176]{Segal_livre}.
\end{proof}

\begin{lem} \label{coh} 
 Let $A$ be virtually cyclic, and $B$ be virtually $F_n$ for some $n$. Up to isomorphism, there are only finitely many groups which are extensions of $A$ by $B$.
\end{lem}

\begin{proof}
 This follows from standard extension theory (\cite{Brown_cohomology}, sections III.10 and IV.6), noting that $\Out(A)$ is finite and $B$ has a finite index subgroup with trivial $H^2$.
\end{proof}

Now consider a McCool group $\M(\calc)\inc\Out(F_3)$. The first step is to reduce to the case where $F_3$ is freely indecomposable relative to $\calc$. If this does not hold, let $\Gamma$ be a   Grushko decomposition relative to $\calc$ (see Subsection \ref{jsj}). It is not unique, we choose one with as few edges as possible. 

If all vertex groups are cyclic, groups 
in $\calc$ are 
generated (up to conjugacy) by powers of elements belonging to  some fixed basis of $F_3$, and finiteness holds. Otherwise,   there is a vertex group $G_v\simeq F_2$. Our choice of $\Gamma$ implies that $\Gamma$ has a single  edge (it is  an HNN extension,  or an amalgam $F_2*\Z$ with a finite index subgroup of $\Z$ belonging to $\calc$). It follows that  $\Gamma$ is $\M(\calc)$-invariant (\cite{For_deformation,Lev_rigid}), and $\M(\calc)$ is determined  by its image in $\Out(F_2)$. This  image is the McCool group $\M(\calc_{ | F_2})$, so finiteness follows from Proposition \ref{infmc2}.  

We continue the proof under the assumption that $F_3$ is freely indecomposable relative to $\calc$. Let $\Gcan$ be the canonical $\M(\calc)$-invariant cyclic JSJ decomposition relative to $\calc$
(see Subsection \ref{jsj}). Vertex groups $G_v$ are cyclic, rigid, or QH.  

One easily checks the formula $\sum_v (\mathrm{rk}\,G_v -1) =2$.
In particular, $\mathrm{rk}\, G_v\leq 3$ for all $v$, and if some $G_v$ is isomorphic to $F_3$, then all other
vertex groups are cyclic.

If $G_v\simeq\pi_1(\Sigma)$ is a QH vertex group, it is  isomorphic to $F_2$ or $F_3$, so  there are 9 possibilities for  the compact surface $\Sigma$: 
\setlist{topsep=0mm,itemsep=.0mm,parsep=.5mm}

\begin{enumerate}
 \item pair of pants
 \item sphere with   4 boundary components
 \item projective plane with 2   boundary components
 \item projective plane with   3 boundary components
  \item torus  with 1  boundary component 
   \item torus   with   2 boundary components
    \item   Klein bottle with 1  boundary component 
     \item   Klein bottle with  2 boundary components
      \item non-orientable surface of genus 3 with 1 boundary component.
\end{enumerate}

Each incident edge group $G_e$ is (up to conjugacy) a boundary subgroup of $\pi_1(\Sigma)$. 
Conversely, there are two possibilities for  a  boundary subgroup $C$. If it is an incident edge group,   it equals   $G_e$ for a unique incident edge. If not, we  say that   the corresponding boundary component of $\Sigma$ is \emph{free}; in this case some finite index subgroup  of $C$ belongs to  $\calc$.

As in Subsection \ref{automs}, the finite index subgroup $\M^0(\calc)$ of $\M(\calc)$ acting trivially on $\Gcan$ maps to $\prod_v \Out(G_v)$ with kernel the group of twists $\calt$. The image in $\Out(G_v)$ is  finite if $G_v$ is cyclic or rigid, virtually the mapping class group  of $\Sigma$ if $G_v$ is QH, and $\calt$ is isomorphic to some $\Z^k$ (see Subsection 4.3 of \cite{GL6}). 

By   mapping class group, we   mean   the group of isotopy classes of homeomorphisms of a compact surface $\Sigma$ mapping each boundary component to itself in an orientation-preserving way. 
We denote it by $\calp\calm^+(\Sigma)$ as in Subsection \ref{jsj}.

By Lemma \ref{vr}, we may assume that there is a QH vertex $v$ with $\calp\calm^+(\Sigma)$ non-solvable. As explained above, there are 9 possibilities for $\Sigma$.
Cases  1, 3, 7 are ruled out  
because $\calp\calm^+(\Sigma)$ is virtually cyclic (see \cite{Szepietowski_presentation}, or argue as in  the proof of    Proposition \ref{infmc2}, noting that  a finite index subgroup of $\calp\calm^+(\Sigma)$ fixes a conjugacy class of $F_2$ which is not a  power of the commutator).

If $\Gcan$ is trivial (i.e.\ if the QH subgroup $G_v $ is the whole group),  
$\M(\calc)$ is the mapping class group of $\Sigma$. We therefore assume that $\Gcan$ is non-trivial.

\begin{lem} \label{ol}
If $G_v$ has rank 3,  
 then  $\Sigma$ has a free boundary component. 
\end{lem}

\begin{proof}
This follows from Lemma 4.1 of \cite{BF_outer}, a generalization of the standard fact that a cyclic amalgam $A*_{\langle c\rangle}B$ of free groups is free only if $c$ belongs to a basis in $A$ or $B$.  
\end{proof}

This lemma   rules out case 9.

Now suppose  that all vertices of $\Gcan$ other than $v$ are terminal vertices carrying $\Z$ (by Lemma \ref{ol}, this holds  in cases 6 and 8). In this case the group of twists $\calt$ is trivial (see Proposition 3.1 of \cite{Lev_automorphisms}). The group $\M(\calc)$ contains $\calp\calm^+(\Sigma)$ with finite index, and there are finitely many possibilities: they depend on whether edges of $\Gcan$ may be permuted, and whether elements in edge groups may be mapped to their inverse.

We must now deal with cases 2, 4, 5. We start with 4. The only possibility left is that $\Gcan$ has two vertices $v,w$ joined by 2 edges, with $G_w$ 
cyclic. 
Every automorphism leaving $\Gcan$ invariant maps $G_v$ to itself (up to conjugacy), and 
we consider the natural map from $\M(\calc)$ to $\Out(G_v)$. As above, the image contains $\calp\calm^+(\Sigma)$ with finite index, and there are finitely many possibilities. The kernel is the group of twists $\calt$, which is isomorphic to $ \Z$. 
 Since $\calp\calm^+(\Sigma)$ is isomorphic to $F_3$ by Theorem 7.5 of \cite{Szepietowski_presentation}, we conclude by Lemma \ref{coh}.

The argument in case 2 is similar. Besides $v$ and $w$, there may be another vertex $w'$, with $G_{w'}$ cyclic and a single edge between $v$ and $w'$. The group $\calp\calm^+(\Sigma)$ is again free, it is isomorphic to $F_2$  (see for instance  \cite{FaMa_primer}, 4.2.4). 

In case 5 (once-punctured torus), there is a single edge incident to $v$. Collapsing all other edges yields an $\M(\calc)$-invariant decomposition as an amalgam $F_3=G_v*_{\langle a\rangle}G_w$ with $G_w\simeq F_2$.  By the standard fact recalled above,    $a$ belongs to a basis of  $G_w$ (and is equal to a commutator in $G_v$).   The group $\M(\calc)$ acts trivially on the graph underlying this amalgam, and
the map $\rho$ (see Subsection \ref{automs}) maps $\M(\calc)$ to $\Out(G_v)\times \Out(G_w)$, with kernel the group of twists $\calt$, isomorphic to $\Z$. The image in $\Out(G_v)$ is isomorphic to $GL(2,\Z)$ or $SL(2,\Z)$.

We now consider the image $L$ of $\M(\calc)$ in $\Out(G_w)$.
It preserves the conjugacy class of $\grp{a}$.
 If $L$ is finite (necessarily of order $\le 6$), then $\M(\calc)$ maps onto $GL(2,\Z)$ or $SL(2,\Z)$ with virtually cyclic kernel $K$; there are finitely many possibilities for $K$ up to isomorphism (it maps to $L$ with cyclic kernel), and we conclude by Lemma \ref{coh}. 
As explained in the proof of Proposition \ref{infmc2}, if $L$ is infinite, it 
is virtually cyclic, contains a ``Dehn twist'' $T_a$, and has index at most 4 in the stabilizer of the conjugacy class of $\grp{a}$ in $\Out(G_w)$.
Since $\M(\calc)$ is determined by its image in $\Out(G_v)\times \Out(G_w)$, and this image contains $SL(2,\Z)\times \langle T_a\rangle$,  this leaves only finitely many possibilities.

\medskip

\begin{flushleft}
 Vincent Guirardel\\
Institut de Recherche Math\'ematique de Rennes\\
Membre de l'institut universitaire de France\\
Universit\'e de Rennes 1 et CNRS (UMR 6625)\\
263 avenue du G\'en\'eral Leclerc, CS 74205\\
F-35042  RENNES C\'edex\\
\emph{e-mail:} \texttt{vincent.guirardel@univ-rennes1.fr}\\[8mm]

Gilbert Levitt\\
Laboratoire de Math\'ematiques Nicolas Oresme\\
Universit\'e de Caen et CNRS (UMR 6139)\\
BP 5186\\
F-14032 Caen Cedex\\
France\\
\emph{e-mail:} \texttt{levitt@unicaen.fr}\\

\end{flushleft}

\end{document}